\documentclass[a4paper,10pt]{amsart}
\usepackage{amsfonts,enumitem}
\usepackage{amssymb,amscd,amsthm,dsfont}
\usepackage{amsmath}
\usepackage[svgnames]{xcolor}
\usepackage{xparse}
\usepackage{listings}
\usepackage{comment}
\usepackage{mathrsfs}

\usepackage[english]{babel}
\usepackage[T1]{fontenc}

\usepackage{todonotes}
\usepackage{etoolbox}
\usepackage{aliascnt}
\usepackage[misc]{ifsym}
\usepackage{bm}
\usepackage{bbm}
\usepackage{Wasysym}
\usepackage[all]{xy} %         ��   ?    ��     ?   ?[all]

\usepackage{stmaryrd}
\usepackage{float}
\usepackage{fancyhdr}
\usepackage{amsxtra,ifthen}
\usepackage{verbatim}

\usepackage[vcentermath]{youngtab}

\usepackage{geometry}
\usepackage{hyperref}

\usepackage{setspace}
\theoremstyle{plain}

\renewcommand{\H}{\mathcal{H}}

\newcommand{\lam}{\lambda}

\newcommand{\bvarnothing}{{\boldsymbol \varnothing}}

\def\bs{{\boldsymbol{s}}}

\def\bxi{{\boldsymbol \xi}}
\def\N{{\mathbb N}}
\def\Z{{\mathbb Z}}
\def\Q{{\mathbb{Q}}}

\def\Lam{\Lambda}

\def\PP{\mathscr{P}}

\def\blam{{\boldsymbol \lambda}}
\def\bmu{{\boldsymbol \mu}}

\def\brho{{\boldsymbol\rho}}
\def\bnu{{\boldsymbol \nu}}

\def\<{\langle}
\def\>{\rangle}

\DeclareMathOperator{\res}{res}

\DeclareMathOperator{\col}{col}

\newcommand{\add}{\operatorname{add}}
\newcommand{\rem}{\operatorname{rem}}

\def\Sym{\mathfrak{S}}

\theoremstyle{plain}

\newtheorem{mainresult}{Theorem}

\csdef{mainresultautorefname}{Theorem}

\newtheorem{theorem}{Theorem}[section]
\newtheorem{lemma}[theorem]{Lemma}
\newtheorem{proposition}[theorem]{Proposition}
\newtheorem{prop}[theorem]{Proposition}

\newtheorem{definition}[theorem]{Definition}
\newtheorem{dfn}[theorem]{Definition}

\newtheorem{example}[theorem]{Example}

\newtheorem{remark}[theorem]{Remark}

\numberwithin{equation}{section}

\newenvironment{equa}
  {\refstepcounter{theorem}%
   \begin{equation}}
  {\end{equation}}

\title{New runner removal theorems for the cyclotomic Hecke algebras of type $G(r,1,n)$}

\author[]{Jun Hu\textsuperscript{1}}
\author[]{Xiangyu Qi\textsuperscript{2,\Letter}}

\thanks{\Letter\, (Corresponding author) Xiangyu Qi,\quad E-mail: qixiangyumath@163.com}

\begin{document}
\maketitle

{\centering{\small \textsuperscript{1}  Key Laboratory of Algebraic Lie Theory and Analysis of Ministry of Education, \\
School of Mathematics and Statistics, Beijing Institute of Technology, Beijing, 100081, P.R.~China\\
E-mail: junhu404@bit.edu.cn\\}}
\vspace{5pt}
	
{\centering{\small \textsuperscript{2} School of Mathematics and Statistics,
	Beijing Institute of Technology,
	Beijing, 100081, P.R.~China\\
E-mail: qixiangyumath@163.com\\}}

\begin{abstract} For the Iwahori-Hecke algebras $\H_q(\Sym_n)$ of the symmetric group $\Sym_n$ at a primitive $e$-th root of unity, James and Mathas proved a theorem which relates $v$-decomposition numbers $d_{\lam\mu}^{e}(v)$ for different values of $e$, by adding ``empty runners'' to the abacus display  for the labelling partitions $\lam,\mu$. Fayers proved a similar theorem, which involves adding ``full'' runners to these abacus displays. In this paper we use Uglov's map from the set of $r$-abaci for $r$-partitions to the set of $e$-abaci for partitions to extend these theorems to the cyclotomic Hecke algebras of type $G(r,1,n)$.
\end{abstract}

\section{Introduction}
The cyclotomic Hecke algebras $\H_{r,n}$ of type $G(r,1,n)$, also known as the Ariki-Koike algebras, were first introduced in the work \cite[Definition 3.1]{AK}, \cite[Definition 4.1]{BM:cyc} and \cite[before Proposition 3.2]{C} as certain deformations of the group ring of the complex reflection group $G(r,1,n)$. They play important roles in the modular representation theory of finite groups of Lie type over fields of non-defining characteristics. These algebras have intimate connection with many different areas of mathematics such as affine Hecke algebras, finite groups of Lie type, and affine quantum groups etc. The representation theory of these algebras have been extensively studied in a number of literatures, say \cite{A1}, \cite{AK}, \cite{AM}, \cite{DJM}, \cite{DM}, \cite{HM}, \cite{LM} and \cite{Ma}.

Among other things, one of the deepest results in this area is Ariki's theorem \cite[Theorem 4.4]{A1} (sometimes also called the Ariki-Lascoux-Leclerc-Thibon's theory), which proved Lascoux-Leclerc-Thibon's conjecture \cite[Conjecture 6.9]{LLT96}. The conjecture asserted that the decomposition matrices of these algebras can be computed in characteristic zero via the evaluation at $q=1$ of the coefficient polynomials $d^{e,\bs}_{\blam,\bmu}(v)$ expressing the canonical bases $G_e^{\bs}(\bmu)$ of the Fock space in terms of its natural basis $s_\blam$, labelled by $r$-partitions. More recent work of Brundan-Kleshchev \cite[Theorem 1.1]{BrundanKleshchev09} shows that each block of these algebras can be identified with the cyclotomic quiver Hecke algebra associated with the affine cyclic quiver $A_{e-1}^{(1)}$ or the line quiver $A_{\infty}$. The latter algebras are $\Z$-graded and play important roles in the categorification theory of quantum groups. Moreover, Brundan-Kleshchev \cite[Theorems 5.3,5.6]{BK:GradedDecomp} showed that the $v$-decomposition numbers $d^{e,\bs}_{\blam,\bmu}(v)$ can be interpreted as the $\Z$-graded decomposition number of the $\Z$-graded Specht module over these cyclotomic quiver Hecke algebras.

These $v$-decomposition numbers $d^{e,\bs}_{\lam,\mu}(v)$ are actually some inverse parabolic Kazhdan-Lusztig polynomials. In the level one (i.e., $r=1$) case, James and Mathas \cite{JM} proved the so called ``empty'' runner removal theorem in which they relate the $v$-decomposition numbers $d^e_{\lam,\mu}(v), d^{e+1}_{\lam^+,\mu^+}(v)$ of Iwahori-Hecke algebras $\H_q(\Sym_n)$ of the symmetric group $\Sym_n$, where the partitions $\lam^+, \mu^+$ are obtained from the partitions $\lam,\mu$ by adding empty runners to the abacus displays for $\lam, \mu$ respectively. After that, Fayers \cite{Fay08} proves a similar theorem, which replaces adding empty runners with adding full runners to these abacus displays. Note that in the same paper \cite{JM} James and Mathas also proved the so called ``empty'' runner removal theorem for the Dipper-James's $q$-Schur algebras \cite{DJ}, using Varagnolo-Vasserot's result \cite[Theorem 6.3]{VV} on $v$-decomposition numbers in place of Ariki's theorem \cite[Theorem 4.4]{A1}.

Many respects the cyclotomic Hecke algebra $\H_{r,n}$ behaves in the same way as the Iwahori-Hecke algebra $\H_q(\Sym_n)$. Many of the combinatorial theorems concerning $\H_q(\Sym_n)$ have been generalised to the cyclotomic Hecke algebra $\H_{r,n}$, with the role of partitions being played by $r$-partitions. For example, the easy definition of $e$-restricted partition was generalized to the so-called ``Kleshchev $r$-partitions'', which has no simple non-recursive characterization. During their study of the generalizations of $e$-cores of partitions to the setting of $r$-partitions. Jacon and Lecouvey \cite{JL} made a very remarkable observation. They introduce a weight-preserving map, defined by Uglov, from the set of $(e,\bs)$-core $r$-partitions to the set of $e$-core partitions. They presented the notions of $(e,\bs)$-core (resp., reduced $(e,\bs)$-core) for all $r$-partition associated with a multicharge (resp., with certain special multicharge). All the theory can be derived from the level one case, except the core of an $r$-partition associated with a multicharge $\bs$ is now itself an $r$-partition associated with a possible different multicharge $\bs'$. This opens a new avenue for studying the modular representation theory of the cyclotomic Hecke algebra $\H_{r,n}$ via the Uglov map. For example, Li and the second author \cite{LQ2025} have used the Uglov map to introduce  new block invariants, called ``block moving vectors'', for the cyclotomic Hecke algebra $\H_{r,n}$, and to classify its representation-finite blocks. Li and Tan \cite{LT} gave new definitions of the $e$-core and $e$-weight of a multipartition associated with an arbitrary multicharge, obtained via its image under the Uglov map. They also generalized the notion of $[w:k]$-pairs to the cyclotomic setting and established a sufficient condition for Scopes equivalence. In \cite{LQT2026}, Li, the second author, and Tan used moving vectors to classify the core blocks of $\H_{r,n}$. This classification yields a criterion for Scopes equivalence between core blocks, interprets their simple module counting as classical Kostka numbers, and relates their graded decomposition numbers in characteristic zero to those of type $A$ Iwahori--Hecke algebras. In \cite{HHLQ}, Hu, Huang, Li, and the second author extended the Uglov map to all classical affine types by uniformly constructing Fock spaces and their combinatorial models (Maya diagrams and abaci). By characterizing the joint highest weight vectors under the Virasoro action, they provided a character-free proof of level-rank duality. In particular, this duality allows the defect of cyclotomic KLR algebras of classical affine types to be expressed as the sum of the components of the corresponding moving vector. Furthermore, this combinatorial framework was recently utilized by Li, Zhang, and Zhu \cite{LZZ} to study cores and Diophantine equations.

In this paper we use the Uglov map to define adding empty runner and adding full runner in the setting of $r$-partitions, see Definitions \ref{lamPlus} and \ref{Plusk}. Then we analyze the image of the canonical bases elements under the induced map on Fock space induced from adding empty runner and adding full runner, see Propositions \ref{KeyCanonicalBases} and \ref{KeyCanonicalBases2}. The following theorem is the first main result of this paper, which generalize James-Mathas's empty runner removal theorem to the cyclotomic Hecke algebra setting.

\begin{mainresult}\label{mainthm1} Fix an integer $0 \le \alpha < e$ and $\bmu\in\PP_{r,n}$. Let $\bs\in \tilde{\mathcal A}^r_e(\bmu)$. Assume that $\bs\in \tilde{\mathcal A}^r_e(\blam)$ whenever $(\blam,\bs)$ and $(\bmu,\bs)$ lie in the same block of $\H_{r,n}^{\Lam_\bs}$.  Let $\bmu^{+}$ be defined using $\bs$ and $\alpha$ as in Definition \ref{lamPlus}. Suppose that $\bmu$ is $(e,\bs)$-regular. Then $\bmu^+$ is $(e+1,\bs)$-regular, and for any $r$-partition $\brho$, $d^{e+1,\bs}_{\brho\bmu^{+}}(v)\neq 0$ only if $\brho=\bnu^{+}$ for some $\bnu\in\PP_{r,n}$. Moreover, $d^{e,\bs}_{\blam \bmu}(v)=d^{e+1,\bs}_{\blam^{+}\bmu^{+}}(v)$ for any $\blam\in\PP_{r,n}$. In particular\footnote{Here in order to avoid the confusion, we use $\hat{S}^{\blam^{+}}$ and $\hat{D}^{\bmu^{+}}$ to denote the Specht module and the simple module over $\H_{r,n^+}^{\Lambda_\bs}$ at a primitive $(e+1)$-th root of unity, where $n^+:=|\bmu^+|$.}, $$
[S^\blam : D^\bmu]=[\hat{S}^{\blam^{+}} : \hat{D}^{\bmu^{+}}].
$$
\end{mainresult}
We refer the readers to Section 2 and Section 3 for unexplained notations here. Note that the cyclotomic Hecke algebra $\H_{r,n}^{\Lam_\bs}$ depends only on $e$ and the multi-set $\{s_j+e\Z\mid 1\leq j\leq r\}$ but not on the multi-set $\{s_j\in\Z\mid 1\leq j\leq r\}$. Given any cyclotomic Hecke algebra $\H_{r,n}^{\Lam_\bs}$, we can always replace certain $s_j$ by $s_j+te$ for some $1\leq j\leq r$ and $t\in\Z$ so that $\bs=(s_1,\cdots,s_r)$ satisfies that $\bs\in \tilde{\mathcal A}^r_e(\blam)$ without changing the algebra $\H_{r,n}^{\Lam_\bs}$ and the modules $S^\blam, D^\bmu$. The following theorem is the second main result of this paper, which generalize Fayers's full runner removal theorem to the cyclotomic Hecke algebra setting.

\begin{mainresult}\label{mainthm2}
Fix $k\in\N, \bs\in \overline{\mathcal A}^r_e$ and an integer $0 \le \alpha < e$. Let $\bmu\in\PP_{r,n}$. Let $\bmu^{+k}$ be defined using $k, \bs$ and $\alpha$ as in Definition \ref{Plusk}. Suppose that $\bmu$ is $(e,\bs)$-regular and $(k-1)e \ge s_r+n$.  Then $\bmu^{+k}$ is $(e+1,\bs^{+k})$-regular, and for any $r$-partition $\brho$, $d^{e+1,\bs^{+k}}_{\brho\bmu^{+k}}(v)\neq 0$ only if $\brho=\bnu^{+k}$ for some $\bnu\in\PP_{r,n}$. Moreover, $d^{e,\bs}_{\blam \bmu}(v)=d^{e+1,\bs^{+k}}_{\blam^{+k}\bmu^{+k}}(v)$ for any $\blam\in\PP_{r,n}$. In particular, $$
[S^\blam : D^\bmu]=[\hat{S}^{\blam^{+k}} : \hat{D}^{\bmu^{+k}}].
$$
\end{mainresult}
Once again, we refer the readers to Section 2 and Section 3 for unexplained notations here. The assumption that $(k-1)e \ge s_r+n$ is not a big deal to us because the cyclotomic Hecke algebra $\H_{r,n}^{\Lam_\bs}$ depends only on $e$ and the multi-set $\{s_j+e\Z\mid 1\leq j\leq r\}$ but not on the multi-set $\{s_j\in\Z\mid 1\leq j\leq r\}$ and we can always replace certain $s_j$ by $s_j+te$ for some $1\leq j\leq r$ and $t\in\Z$ so that $\bs=(s_1,\cdots,s_r)$ satisfies that $\bs\in \overline{\mathcal A}^r_e$ and $(k-1)e \ge s_r+n$ without changing the algebra $\H_{r,n}^{\Lam_\bs}$ and the modules $S^\blam, D^\bmu$. We also note that there are some other generalizations of James-Mathas's and Fayers's runner removal theorems to $\H_{r,n}$ in the literatures, say, \cite{DellA24b}, \cite{DP2026}, \cite{Q}. However, their definitions of adding empty runners or full runners are directly on the $e$-abaci of each component partition of the given $r$-partitions, which do not use the Uglov map and hence is different with our definitions of adding empty runners or full runners.

The content of the paper is organised as follows. In Section 2 we collect some preliminary results on the combinatorics of $r$-partitions and their abacus display. In particular, we recall the definitions of $e$-tuple column abacus for partitions and the $(e,\bs)$ row abaci (or $r$-abaci) for $r$-partitions. We introduce the Uglov maps which are maps from the configuration of the abaci display for all the pair $(\blam,\bs)$ to the configuration of the abaci display for all the pair $(\lam,s)$, where $\bs\in\Z^r, s\in\Z$, and $\blam, \lam$ are $r$-partition and partition respectively. In Section 3 we first introduce the definition of adding empty runner on abacus of $r$-partitions via Uglov's map, then we recall some basic result on Fock space over quantum affine algebras and the Ariki-Lascoux-Leclerc-Thibon theory on $v$-decomposition. In Section 4, we first show that adding an empty runner for $r$-partition send the set of $(e,\bs)$-regular $r$-partitions to $(e+1,\bs)$-regular $r$-partitions.. We then prove the first main result in this paper, namely the empty runner removal theorem for the cyclotomic Hecke algebra \(\H_{r,n}\). In Section 5 we first introduce the definition of adding full runner on abacus of $r$-partitions via Uglov's map, give some characterization on the addable and removable nodes on $L_{\bs^{+k}}(\blam^{+k})$, then we show that adding a full runner for $r$-partition send the set of $(e,\bs)$-regular $r$-partitions to $(e+1,\bs^{+k})$-regular $r$-partitions. Finally, we give the proof of the second main result in this paper, namely the full runner removal theorem for the cyclotomic Hecke algebra \(\H_{r,n}\).

\bigskip
\centerline{Acknowledgements}
\smallskip

The research was supported by the National Natural Science Foundation of China (No. 12431002).
\bigskip

\section{Preliminaries}
In this section, we shall collect some preliminary results on the $r$-partitions and their abaci display.

\subsection{$r$-partitions and abaci}

Let $n$ be a non-negative integer. A {\em partition} $\lam$ of $n$ is a non-increasing sequence of non-negative integers $\lam=(\lam_1,\dots,\lam_s)$ such that $\sum_{i=1}^{s}\lam_i=n$ and we write $|\lam|=n$.
The {\em Young diagram} of a partition $\lam$ is the set of nodes $[\lam]=\{(i,j)\mid i\geq 1, 1\leq j\leq\lam_{i}\}$.We denote by $\mathcal{P}_n$ the set of partitions  of $n$. The integers $\lambda_b$, for $b\ge 1$, are the \textit{parts} of $\lambda$. For any partition $\lambda$, we write $\ell(\lambda)$ for the number of non-zero parts of $\lambda$, and we call it the \textit{length} of $\lambda$.
The {\em conjugate} of $\lam$ is defined to be the partition $\lam'=(\lam'_1,\lam'_2,\cdots)$, where $\lam'_k:=\#\{j\geq 1\mid\lam_j\geq k\}$ for $j=1,2,\cdots$.

%A {\em rim $e$-hook} (or simply
%an $e$-rim) of $[\lam]$ is a connected subset of the rim of $[\lam]$ with exactly $e$ nodes, which can be
%removed from $[\lam]$ to obtain another Young diagram $[\mu]$.
%Given a partition $\lam$,
%unwrapping a rim $e$-hook of $[\lam]$ one by one until none can be unwrapped, one can get a partition, called
%the $e$-{\em core} of $[\lam]$ and the number of rim $e$-hooks unwrapped is called the $e$-{\em weight} of $\lam$.

Let $r\in\Z_{\geq 1}$. An {\em $r$-partition} of $n$ is an $r$-tuple $\blam=(\blam^{(1)}, \dots, \blam^{(r)})$ of partitions such that $|\blam|:=\sum_{i=1}^r|\blam^{(i)}|=n$. The partitions $\blam^{(1)}, \dots,
\blam^{(r)}$ are called components of $\blam$. The conjugate of an $r$-partition $\blam$ is defined to be $\blam'=(\blam^{(r)'}, \dots, \blam^{(1)'})$. For $\sigma\in\mathfrak{S}_r$, where $\mathfrak{S}_r$ is the symmetric group on $\{1, \dots, r\}$, define $\blam^\sigma$ to be $(\blam^{(\sigma(1))}, \dots, \blam^{(\sigma(r))})$.
We denote the set of $r$-partitions of $n$ by $\mathscr{P}_{r,n}$.

We recall some basic definitions and facts about abaci display for partitions and $r$-partitions. By a ``runner'' we mean a horizontal or vertical line labelled by the integers \(\mathbb{Z}\), ordered from left to right (respectively from top to bottom). By an abacus display (or representation) of a subset $B$ of $\Z$ we mean putting a bead at the position $a$ of the runner for each $a\in B$. For the abacus representation of partitions, the fundamental reference is James--Kerber's book \cite{JK}. For the abacus representation of $r$-partitions, we refer the readers to \cite{JL} and \cite{LQ2025}. We always insert a dashed vertical line separating positions \(-1\) and 0, and we generally omit the numbers printed under each position. Given a partition $\lam$ and $s\in \mathbb{Z}$, one can associate it to a set of integers $L_s(\lam)=\{\lam_j-j+s\mid j\in\mathbb{N}^+\}$.
We assume that $\lam$ has an infinite number of zero parts. As is well-known, the set $L_s(\lam)$ can be expressed by a {\em row abacus}.
Let us illustrate an example. \begin{example}
Set \(\lambda=(7,5,4,1,1)\) and \(s=0\). For each \(i\in L_s(\lambda)\), we place a bead on the $i$-th position of the horizontal row abacus. The collection \(L_s(\lambda)\) is shown below. Positions with no bead are termed empty positions.
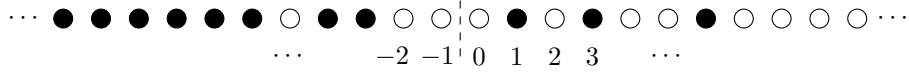
\[
\begin{tikzpicture}
[scale=0.5, bb/.style={draw,circle,fill,minimum size=2.5mm,inner sep=0pt,outer sep=0pt},
wb/.style={draw,circle,fill=white,minimum size=2.5mm,inner sep=0pt,outer sep=0pt}]
\foreach \x in {13,-10}
\foreach \y in {2}
{
\node at (\x,\y) {$\cdots$};
}	
\node [] at (2,1) {$0$};
\node [] at (3,1) {$1$};
\node [] at (4,1) {$2$};
\node [] at (5,1) {$3$};
\node [] at (7,1) {$\cdots$};
\node [] at (0.9,1) {$-1$};
\node [] at (-0.3,1) {$-2$};
\node [] at (-3,1) {$\cdots$};
\node [wb] at (12,2) {};
\node [wb] at (11,2) {};
\node [wb] at (10,2) {};
\node [wb] at (9,2) {};
\node [bb] at (8,2) {};
\node [wb] at (7,2) {};
\node [wb] at (6,2) {};
\node [bb] at (5,2) {};
\node [wb] at (4,2) {};
\node [bb] at (3,2) {};
\node [wb] at (2,2) {};
\node [wb] at (1,2) {};
\node [wb] at (0,2) {};
\node [bb] at (-1,2) {};
\node [bb] at (-2,2) {};
\node [wb] at (-3,2) {};
\node [bb] at (-4,2) {};
\node [bb] at (-5,2) {};
\node [bb] at (-6,2) {};
\node [bb] at (-7,2) {};
\node [bb] at (-8,2) {};
\node [bb] at (-9,2) {};
\draw[dashed](1.5,1)--node[]{}(1.5,2.5);
\end{tikzpicture}.
\]
Note that the charge $0$ is equal to the number of beads to the right of the dashed line minus the number of empty positions to the left of the dashed line.
\end{example}

Let $1<e<\infty$ be a fixed integer. A row abacus $L_s(\lam)$ can also be represented by an $e$-tuple of column abaci, where the column abaci are indexed \(L_0, L_1, \dots, L_{e-1}\) from left to right. For each $k\in L_s(\lam)$, if $k=ye+x$ with $x, y\in\mathbb{Z}$ and $0\leq x<e$, then place a bead in position $(x, y)$, which means the $y$-th position of $L_x$. We will denote this $e$-tuple column abacus by $\mathcal{L}_s^e(\lam)$. We usually label the runners of the $e$-tuple column abacus of $\mathcal{L}_s^e(\lam)$ as $\rho_0,\rho_1,\cdots,\rho_{e-1}$ from left to right.

\begin{example}
Let $e=3$,  $s=0$ and $\lambda=(7, 5, 4 ,1, 1)$. Then $\mathcal{L}_s^e(\lam)$ is as follows.
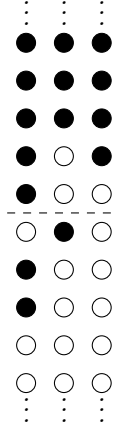
\begin{center}
\begin{tikzpicture}[scale=0.5, bb/.style={draw,circle,fill,minimum size=2.5mm,inner sep=0pt,outer sep=0pt}, wb/.style={draw,circle,fill=white,minimum size=2.5mm,inner sep=0pt,outer sep=0pt}]
\foreach \x in {-1,0,1}
\foreach \y in {6,-4.5}
{
\node at (\x,\y) {$\vdots$};
}
\node[bb] at (-1, 5){};
\node[bb] at (-1, 4){};
\node[bb] at (-1, 3){};
\node[bb] at (-1, 2){};
\node[bb] at (-1, 1){};
\node[wb] at (-1, 0){};
\node[bb] at (-1, -1){};
\node[bb] at (-1, -2){};
\node[wb] at (-1, -3){};
\node[wb] at (-1, -4){};

\node[bb] at (0, 5){};
\node[bb] at (0, 4){};
\node[bb] at (0, 3){};
\node[wb] at (0, 2){};
\node[wb] at (0, 1){};
\node[bb] at (0, 0){};
\node[wb] at (0, -1){};
\node[wb] at (0, -2){};
\node[wb] at (0, -3){};
\node[wb] at (0, -4){};

\node[bb] at (1, 5){};
\node[bb] at (1, 4){};
\node[bb] at (1, 3){};
\node[bb] at (1, 2){};
\node[wb] at (1, 1){};
\node[wb] at (1, 0){};
\node[wb] at (1, -1){};
\node[wb] at (1, -2){};
\node[wb] at (1, -3){};
\node[wb] at (1, -4){};

\draw[dashed](-1.5,0.5)--node[]{}(1.5,0.5);
	\end{tikzpicture}
	\end{center}
\end{example}

Any $r$-tuple $\bs=(s_1, s_2, \dots, s_r)\in \mathbb{Z}^r$ is called a \emph{multicharge}. Let $\bs=(s_1, s_2, \dots, s_r)$ be a given multicharge. Let $\blam=(\blam^{(1)}, \dots, \blam^{(r)})$ be any given $r$-partition.
Then the pair $(\blam, \bs)$ can be associated with an $r$-abacus $L_\bs(\blam)$ by setting the row abaci $L_{s_i}(\blam^{(i)})$, $i=1, \dots, r$, from bottom to top so that all positions $0$ of each abacus appear in the same vertical line.  We always insert a dashed vertical line separating positions \(-1\) and $0$. We also call the $r$-abacus $L_{\bs}(\blam)$ the $(e,\bs)$ row abacus (or $(e,\bs)$-abacus for short) of the pair $(\blam, \bs)$. If $r=1$, then the $1$-abacus is nothing but the row abacus of a partition we introduced before.

Conversely, given any $r$-abacus diagram $L$ with a dashed vertical line separating positions \(-1\) and $0$, if, for each \(1 \leq i \leq r\), both the number of beads on \(L_i\) lying to the right of the dashed line and the number of empty positions on \(L_i\) lying to the left of the dashed line are finite, then we may attach a multicharge $\bs=(s_1, s_2, \cdots, s_r)\in\mathbb{Z}^r$ to it as follows:
\begin{equa}\label{Biject1}
s_j:=\begin{pmatrix}\text{the number of beads in $L$ to the}\\
\text{right of the dashed vertical line}\end{pmatrix}-\begin{pmatrix}\text{the number of empty positions beads in $L$}\\
\text{to the left of the dashed vertical line}\end{pmatrix},
\quad\forall\,1\leq j\leq r
\end{equa}
and we can also read off an $r$-partition $\blam$ from $L$, such that $L$ is the $r$-abacus diagram of the pair \((\blam,\bs)\). The above correspondence is a mutual bijection.

\begin{example} Let $e:=3, r:=4, \blam:=(2,1), (3, 2), (1, 1), \varnothing)$, $\bs:=(0, 0, 2, 2)$.
 The $(4,\bs)$-abacus $L_\bs(\blam)$  of $(\blam,\bs)$ is given as follows: 
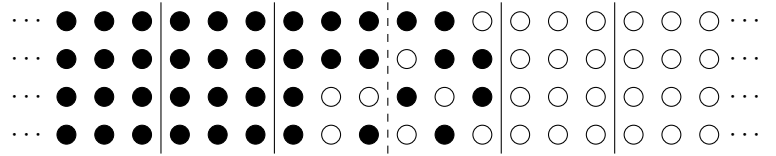
\begin{center}
\begin{tikzpicture}[scale=0.5, bb/.style={draw,circle,fill,minimum size=2.5mm,inner sep=0pt,outer sep=0pt}, wb/.style={draw,circle,fill=white,minimum size=2.5mm,inner sep=0pt,outer sep=0pt}]
	
\foreach \x in {10,-9}
\foreach \y in {3,2,1,0}
{
\node at (\x,\y) {$\cdots$};
}	
	\node [wb] at (9,3) {};
	\node [wb] at (8,3) {};
	\node [wb] at (7,3) {};
	\node [wb] at (6,3) {};
	\node [wb] at (5,3) {};
	\node [wb] at (4,3) {};
	\node [wb] at (3,3) {};
	\node [bb] at (2,3) {};
	\node [bb] at (1,3) {};
	\node [bb] at (0,3) {};
	\node [bb] at (-1,3) {};
	\node [bb] at (-2,3) {};
	\node [bb] at (-3,3) {};
	\node [bb] at (-4,3) {};
	\node [bb] at (-5,3) {};
	\node [bb] at (-6,3) {};
	\node [bb] at (-7,3) {};
	\node [bb] at (-8,3) {};

	\node [wb] at (9,2) {};
	\node [wb] at (8,2) {};
	\node [wb] at (7,2) {};
	\node [wb] at (6,2) {};
	\node [wb] at (5,2) {};
	\node [wb] at (4,2) {};
	\node [bb] at (3,2) {};
	\node [bb] at (2,2) {};
	\node [wb] at (1,2) {};
	\node [bb] at (0,2) {};
	\node [bb] at (-1,2) {};
	\node [bb] at (-2,2) {};
	\node [bb] at (-3,2) {};
	\node [bb] at (-4,2) {};
	\node [bb] at (-5,2) {};
	\node [bb] at (-6,2) {};
	\node [bb] at (-7,2) {};
	\node [bb] at (-8,2) {};
	
	\node [wb] at (9,1) {};
	\node [wb] at (8,1) {};
	\node [wb] at (7,1) {};
	\node [wb] at (6,1) {};
	\node [wb] at (5,1) {};
	\node [wb] at (4,1) {};
	\node [bb] at (3,1) {};
	\node [wb] at (2,1) {};
	\node [bb] at (1,1) {};
	\node [wb] at (0,1) {};
	\node [wb] at (-1,1) {};
	\node [bb] at (-2,1) {};
	\node [bb] at (-3,1) {};
	\node [bb] at (-4,1) {};
	\node [bb] at (-5,1) {};
	\node [bb] at (-6,1) {};
	\node [bb] at (-7,1) {};
	\node [bb] at (-8,1) {};
	
	\node [wb] at (9,0) {};
	\node [wb] at (8,0) {};
	\node [wb] at (7,0) {};
	\node [wb] at (6,0) {};
	\node [wb] at (5,0) {};
	\node [wb] at (4,0) {};
	\node [wb] at (3,0) {};
	\node [bb] at (2,0) {};
	\node [wb] at (1,0) {};
	\node [bb] at (0,0) {};
	\node [wb] at (-1,0) {};
	\node [bb] at (-2,0) {};
	\node [bb] at (-3,0) {};
	\node [bb] at (-4,0) {};
	\node [bb] at (-5,0) {};
	\node [bb] at (-6,0) {};
	\node [bb] at (-7,0) {};
	\node [bb] at (-8,0) {};
	
	\draw[](-5.5,-0.5)--node[]{}(-5.5,3.5);
	\draw[](-2.5,-0.5)--node[]{}(-2.5,3.5);
	\draw[dashed](0.5,-0.5)--node[]{}(0.5,3.5);
		\draw[](3.5,-0.5)--node[]{}(3.5,3.5);
	\draw[](6.5,-0.5)--node[]{}(6.5,3.5);
	\end{tikzpicture}.
\end{center}
\end{example}

Let $1\leq i\leq r$ and $j\in\Z$. Denote by $\CIRCLE_j^i(\blam, \bs)$ the $j$-th bead on the abacus $L_{s_i}(\blam^{(i)})$, counted from right to left. If there is no danger of confusion, we write it simply as $\CIRCLE_j^i$.
The following lemma is straightforward to verify.

\begin{lemma}\text{\rm (\cite[Lemm 3.1.8]{LQ2025})}\label{bead minus}
Let $L_\bs(\blam)$ be the $(e,\bs)$-abacus of the pair $(\blam, \bs)$. Let $1\leq i\leq r$ and $j\in\Z$. Then the number of empty positions to the left of $\CIRCLE_j^i$ is equal to $\blam^{(i)}_j$.
\end{lemma}

The Young diagram of an $r$-partition $\blam$ is the set of nodes $$[\blam]:=\{(i, j, k)\mid i\geq 1, 1\leq j\leq\blam^{(k)}_i, 1\leq k\leq r\}.$$
Let $e\in\Z_{>1}\cup\{0\}$. Given $\bs=(s_1, s_2, \dots, s_r)\in \Z^r$, define the {\em residue of the node} $\gamma=(i, j, k)\in [\blam]$ to be $\res(\gamma):={j-i+s_k}+e\Z\in\Z/e\Z$, and we say $\gamma$ is a $\res(\gamma)$-node. For any $f\in\Z/e\Z$, we define
$c_f(\blam)$ to be the number of $f$-nodes in $[\blam]$ and $\mathcal{C}_{e, \bs}(\blam):=(c_0(\blam), c_1(\blam), \dots, c_{{e-1}}(\blam))$.
Let $\blam\in\PP_{r,n}, i\in I$ and let $\gamma$ be an $i$-node. If $\gamma\notin [\blam]$ and $[\blam]\cup\{\gamma\}$ is the Young diagram of an $r$-partition, then we say $\gamma$ is an addable $i$-node of $\blam$; if $\gamma\in [\blam]$ and $[\blam]\setminus\{\gamma\}$ is the Young diagram of an $r$-partition, then we say $\gamma$ is an removable $i$-node of $\blam$. We use $\add_i\blam$ (resp., $\rem_i\bmu$) to denote the set of addable (resp., removable) $i$-nodes of $\blam$.

\begin{lemma}\text{\rm (\cite[Lemma 3.2.2]{LQ2025})}\label{3.2.2}
Let $\lambda$ be a partition and $s\in\Z$. If $\CIRCLE_i$ is at $(j+ke)$-th position in $L_s(\lam)$, where $0\le j\le e-1$,
then the residue of node $(i, \lam_i)$ is $j$.
\end{lemma}

\begin{lemma}\label{add and rem}
An addable $i$-node of $(\boldsymbol{\lambda},\boldsymbol{s})$ corresponds to the configuration of the $(e,\bs)$-abacus $L_{\boldsymbol{s}}(\boldsymbol{\lambda})$ such that position $(a,b-1)$ is occupied by a bead while $(a,b)$ is empty, where $1\le a\le r, b\in\Z$, and $b\equiv i\pmod{e}$. Adding this addable $i$-node corresponds to moving the bead at position \((a,b-1)\) to \((a,b)\) on \(L_{\boldsymbol{s}}(\boldsymbol{\lambda})\). A removable $i$-node of \((\boldsymbol{\lambda},\bs)\) corresponds to the configuration of the $(e,\bs)$-abacus \(L_{\boldsymbol{s}}(\boldsymbol{\lambda})\) such that position \((a,b-1)\) is empty while \((a,b)\) carries a bead, where \(1\le a\le r\), \(b\in\mathbb{Z}\), and \(b\equiv i\pmod{e}\). Removing this removable $i$-node corresponds to moving the bead at position \((a,b)\) to \((a,b-1)\) on \(L_{\boldsymbol{s}}(\boldsymbol{\lambda})\).
\end{lemma}

Let $A$ be a finite dimensional algebra over a field $K$. A left $A$-module $M$ is said to belong to some block of $A$ if all the composition factors of $M$ belong to that block. For any cellular algebra $A$,
and any cell module $M$ over $A$, according to \cite[(C3),(C3)']{GL}, any central primitive idempotent acts on $M$ by a scalar. It follows that all composition factors of $M$ belong to the same block.
Therefore, we can say in this case a cell module belongs to some block.

Let $1\neq\xi\in K$ be an invertible element which has quantum characteristic $e$. That says, either $e=0$ and $\xi$ is not a root of unity, or $e>1$ and $\xi$ is a primitive $e$-th root of unity in $K^\times$.
Let $\mathfrak{g}:=\hat{\mathfrak{sl}}_e$ if $e>1$; or $\mathfrak{g}:={\mathfrak{sl}}_\infty$ if $e=0$.
Set $\Lam_{\bs}:=\Lam_{\overline{s_1}}+\cdots+\Lam_{\overline{s_r}}$, which is a dominant weight of $\mathfrak{g}$, where $\overline{s_j}:=s_j+e\Z\in\Z/e\Z$ for each $j$. Let $\{\alpha_i\mid i\in I\}$ be the set of simple roots of $\mathfrak{g}$. Let $\H_{r,n}^{\Lam_\bs}$ be the cyclotomic Hecke algebra of type $G(r,1,n)$ with Hecke parameter $\xi$ and cyclotomic parameters $\xi^{s_1},\cdots,\xi^{s_r}$. By definition, $\H_{r,n}^{\Lam_\bs}$ is the unital $K$-algebra generated by $T_0,T_1,\ldots,T_{n-1}$, subject to the following relations: $$\begin{aligned}
&  \prod_{j=1}^r(T_0-\xi^{\kappa_j})=0,\quad (T_a-\xi)(T_a+1)=0,\,\,\forall\, 1\leq a<n,\\
& T_0T_1T_0T_1=T_1T_0T_1T_0, \quad T_aT_b=T_bT_a,\,\,\forall\,0\leq a<b-1<n-1,\\
& T_aT_{a+1}T_a=T_{a+1}T_aT_{a+1},\quad\forall\,1\leq a<n-1 .\end{aligned} $$

By \cite{DJM}, $\H_{r,n}^{\Lam_\bs}$ is a cellular algebra over $K$. By \cite{LM}, the blocks of $\H_{r,n}^{\Lam_\bs}$ are parameterized by some $\alpha\in Q_n^+$, where $Q_n^+$ is the set of elements in the positive root lattice of $\mathfrak{g}$ with height $n$. The Dipper-James-Mathas's Specht module $S^\blam$ of $\H_{r,n}^{\Lam_\bs}$ lies in the block corresponding to $\beta\in Q_n^+$ if and only if $\beta=\sum_{i\in I}c_{i}(\blam)\alpha_i$. In this case, we often abuse notation to say that $\blam$ lies in the block $\mathcal {H}_{\beta}^{\Lambda}$ of $\H_{r,n}^\Lam$. Since the definition of a cell module relies on the chosen $r$-tuple \(\boldsymbol{s}\), we shall write the pair \((\boldsymbol{\lambda}, \boldsymbol{s})\) within each block whenever clarification is needed.

Given two pairs $(\blam, \bs)$ and $(\bmu, \bs)$, the following lemma provides a criterion of being in the same block.

\begin{lemma}\text{\rm (\cite[Theorem 2.11]{LM})}\label{residueblock}
Pairs $(\blam, \bs)$ and $(\bmu, \bs)$ belong to the same block if and only of $c_f(\blam)=c_f(\bmu)$ for all $f\in K$.
\end{lemma}

Based on Lemmas \ref{residueblock}, for any $\sigma\in \mathfrak{S}_r$, pairs $(\blam^\sigma, \bs^{\sigma})$ and $(\blam, \bs)$ are in the same block.

\smallskip
\subsection{Uglov maps}

\begin{definition}\text{\rm (\cite[Section 4.1]{U}, \cite[Section 2.4]{JL})}\label{Uglovmap}
Let $\bs=(s_1, s_2, \dots, s_r)\in \mathbb{Z}^r$ be a multicharge. Let $\blam$ be an $r$-partition. Then the image of pair $(\blam, \bs)$ under Uglov map $\tau_{e, r}$ is $(\lam, s)$, where $s:=\sum_{j=1}^r s_j$,
which is defined as follows. For each bead at position $(x, y)$ in $L_\bs(\blam)$, where $x\in\{1,2,\cdots,r\}$, let $y = k.e + c$ with $k \in \Z$ and $c \in \{0, \dots, e-1\}$. Then we set a bead in the new $1$-abacus $L_s(\lam)$ at position $(r-x)e+ker +c$.
\end{definition}

Using the $e$-tuple column abacus display for partitions, we can give an intuitive description of the Uglov map $\tau_{e, r}$ as follows. Partition the $r$-abacus \(L_\bs(\lambda)\) into segments consisting of positions \(\{ae, ae+1, \dots, ae+e-1\}\) for each integer $a$. For every such segment indexed by $a$, place it above the segment corresponding to \((a+1)e\), which occupies positions \(\{(a+1)e, (a+1)e+1, \dots, (a+1)e+e-1\}\). The resulting configuration is the $e$-tuple column abacus diagram for \(\tau_{e,r}\big(L_s(\lambda)\bigr)\). We provide an illustrative example below.

\begin{example}
Let $e=3$, $r=4$, $\bs=(0,0,2,2)$ and $\blam=((2), (3,1),(1,1),\emptyset)$. Then the $(4,\bs)$-abacus $L_\bs(\blam)$  of the pair $(\blam,\bs)$ is as follows.
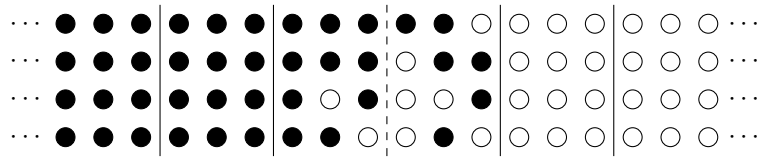
\begin{center}
\begin{tikzpicture}[scale=0.5, bb/.style={draw,circle,fill,minimum size=2.5mm,inner sep=0pt,outer sep=0pt}, wb/.style={draw,circle,fill=white,minimum size=2.5mm,inner sep=0pt,outer sep=0pt}]
	
\foreach \x in {10,-9}
\foreach \y in {3,2,1,0}
{
\node at (\x,\y) {$\cdots$};
}	
	\node [wb] at (9,3) {};
	\node [wb] at (8,3) {};
	\node [wb] at (7,3) {};
	\node [wb] at (6,3) {};
	\node [wb] at (5,3) {};
	\node [wb] at (4,3) {};
	\node [wb] at (3,3) {};
	\node [bb] at (2,3) {};
	\node [bb] at (1,3) {};
	\node [bb] at (0,3) {};
	\node [bb] at (-1,3) {};
	\node [bb] at (-2,3) {};
	\node [bb] at (-3,3) {};
	\node [bb] at (-4,3) {};
	\node [bb] at (-5,3) {};
	\node [bb] at (-6,3) {};
	\node [bb] at (-7,3) {};
	\node [bb] at (-8,3) {};

	\node [wb] at (9,2) {};
	\node [wb] at (8,2) {};
	\node [wb] at (7,2) {};
	\node [wb] at (6,2) {};
	\node [wb] at (5,2) {};
	\node [wb] at (4,2) {};
	\node [bb] at (3,2) {};
	\node [bb] at (2,2) {};
	\node [wb] at (1,2) {};
	\node [bb] at (0,2) {};
	\node [bb] at (-1,2) {};
	\node [bb] at (-2,2) {};
	\node [bb] at (-3,2) {};
	\node [bb] at (-4,2) {};
	\node [bb] at (-5,2) {};
	\node [bb] at (-6,2) {};
	\node [bb] at (-7,2) {};
	\node [bb] at (-8,2) {};
	
	\node [wb] at (9,1) {};
	\node [wb] at (8,1) {};
	\node [wb] at (7,1) {};
	\node [wb] at (6,1) {};
	\node [wb] at (5,1) {};
	\node [wb] at (4,1) {};
	\node [bb] at (3,1) {};
	\node [wb] at (2,1) {};
	\node [wb] at (1,1) {};
	\node [bb] at (0,1) {};
	\node [wb] at (-1,1) {};
	\node [bb] at (-2,1) {};
	\node [bb] at (-3,1) {};
	\node [bb] at (-4,1) {};
	\node [bb] at (-5,1) {};
	\node [bb] at (-6,1) {};
	\node [bb] at (-7,1) {};
	\node [bb] at (-8,1) {};
	
	\node [wb] at (9,0) {};
	\node [wb] at (8,0) {};
	\node [wb] at (7,0) {};
	\node [wb] at (6,0) {};
	\node [wb] at (5,0) {};
	\node [wb] at (4,0) {};
	\node [wb] at (3,0) {};
	\node [bb] at (2,0) {};
	\node [wb] at (1,0) {};
	\node [wb] at (0,0) {};
	\node [bb] at (-1,0) {};
	\node [bb] at (-2,0) {};
	\node [bb] at (-3,0) {};
	\node [bb] at (-4,0) {};
	\node [bb] at (-5,0) {};
	\node [bb] at (-6,0) {};
	\node [bb] at (-7,0) {};
	\node [bb] at (-8,0) {};
	
	\draw[](-5.5,-0.5)--node[]{}(-5.5,3.5);
	\draw[](-2.5,-0.5)--node[]{}(-2.5,3.5);
	\draw[dashed](0.5,-0.5)--node[]{}(0.5,3.5);
		\draw[](3.5,-0.5)--node[]{}(3.5,3.5);
	\draw[](6.5,-0.5)--node[]{}(6.5,3.5);
	\end{tikzpicture}
\end{center}
The $e$-tuple abacus of $\tau_{e, r}(L_{\bs}(\blam))$ is
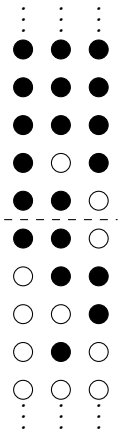
\begin{center}
\begin{tikzpicture}[scale=0.5, bb/.style={draw,circle,fill,minimum size=2.5mm,inner sep=0pt,outer sep=0pt}, wb/.style={draw,circle,fill=white,minimum size=2.5mm,inner sep=0pt,outer sep=0pt}]
\foreach \x in {-1,0,1}
\foreach \y in {6,-4.5}
{
\node at (\x,\y) {$\vdots$};
}
\node[bb] at (-1, 5){};
\node[bb] at (-1, 4){};
\node[bb] at (-1, 3){};
\node[bb] at (-1, 2){};
\node[bb] at (-1, 0){};
\node[bb] at (-1, 1){};
\node[wb] at (-1, -1){};
\node[wb] at (-1, -2){};
\node[wb] at (-1, -3){};
\node[wb] at (-1, -4){};

\node[bb] at (0, 5){};
\node[bb] at (0, 4){};
\node[bb] at (0, 3){};
\node[wb] at (0, 2){};
\node[bb] at (0, 1){};
\node[bb] at (0, 0){};
\node[bb] at (0, -1){};
\node[wb] at (0, -2){};
\node[bb] at (0, -3){};
\node[wb] at (0, -4){};

\node[bb] at (1, 5){};
\node[bb] at (1, 4){};
\node[bb] at (1, 3){};
\node[bb] at (1, 2){};
\node[wb] at (1, 1){};
\node[wb] at (1, 0){};
\node[bb] at (1, -1){};
\node[bb] at (1, -2){};
\node[wb] at (1, -3){};
\node[wb] at (1, -4){};

\draw[dashed](-1.5,0.5)--node[]{}(1.5,0.5);
	\end{tikzpicture}
	\end{center}
\end{example}

\begin{lemma}\label{Uglov} The Uglov map \(\tau_{e,r}\) gives a bijection between the $r$-abacus diagram set \(\bigsqcup_{\substack{\bs\in\mathbb{Z}^r\\ \blam\in\mathscr{P}_{r,n},\, n\in\mathbb{N}}}L_{\bs}(\blam)\)
associated with $r$-partitions and its counterpart \(\bigsqcup_{\substack{s\in\mathbb{Z},\, n\in\mathbb{N}\\ \mu\in\mathcal{P}_n}}L_s(\mu)\) for ordinary partitions.
\end{lemma}
\begin{proof} This is clear.
\end{proof}
Note that the Uglov map defined above does not induce a bijection between the set of $r$-partitions and the set of partitions. For example, take $e=3$, $r=2$, then $\tau_{2,3}$ sends $\bigl((\emptyset,\emptyset), (0,3)\bigr)$ to $(\emptyset,3)$, and sends $\bigl((\emptyset,\emptyset), (2,1)\bigr)$ to
$((2^2),3)$, while $\emptyset\neq (2^2)$.

Following \cite{JL}, we introduce the following two subsets of $\mathbb{Z}^r$. $$\begin{aligned}
\overline{\mathcal A}^r_e:&=\{ (s_1, \dots, s_r) \in \mathbb{Z}^r \mid \forall\, i,\, j\in \{1,\dots,r\}, i<j,\, 0\le s_j-s_i\le e\},\\
{\mathcal A}^r_e:&=\{(s_1,\dots,s_r)\in \mathbb{Z}^r \mid \forall\, i,\, j\in \{1,\dots,r\}, i<j, \,0\le s_j-s_i<e\}.\end{aligned}
$$
Clearly, ${\mathcal A}^r_e\subset\overline{\mathcal A}^r_e$. Let $\blam\in\mathscr{P}_{r,n}$. We define
$$
\widetilde{\mathcal A}^r_e (\blam):=\overline{\mathcal A}^r_e \cap \{(s_1, \cdots, s_r) \in \mathbb Z^r \mid \forall j \in \{1, \cdots, r\}, s_j \ge \ell (\lambda^{(j)})\}
$$

\begin{lemma}\label{s_jge}
For any $1\le j\le r$, $s_j \ge \ell(\lambda^{(j)})$ if and only if all positions on $L_\bs (\blam)$ to the left of the dashed line are occupied by beads, which is equivalent to the statement that all positions above the dashed line on $\tau_{e, r}(L_\bs(\blam))$ are occupied by beads.
%$L_\bs (\blam)$           ?      , which is equivalent to $\tau_{e, r}(L_\bs(\blam))$         ?         .
\end{lemma}

\begin{lemma}\label{adjustCharge} Suppose the $r$-partitions \(\boldsymbol{\lambda}\) and \(\boldsymbol{\mu}\) lie in the same block. Then there always exists some \(\boldsymbol{s} \in \widetilde{\mathcal A}^r_e (\boldsymbol{\lambda})\) such that \(\boldsymbol{s} \in \widetilde{\mathcal A}^r_e (\boldsymbol{\mu})\).
%$\blam, \bmu$     ??    .  ?    $\bs \in \widetilde{\mathcal A}^r_e (\blam)$, ?   $\bs \in \widetilde{\mathcal A}^r_e (\bmu)$.
\end{lemma}
\begin{proof} If \(\boldsymbol{s}\) fails to satisfy the above assumption stated in the lemma, we may simply replace \(\boldsymbol{s}\) with \(\boldsymbol{s}' = (s_1+k, \dots, s_r+k)\) for $k\gg 0$. Since each block contains a finite number of $r$-partitions, the lemma follows at once.
\end{proof}

\bigskip
\section{Adding empty runner for $r$-partitions and Fock space over quantum affine algebra}

The purpose of this section is to introduce our definition of adding empty runner on abacus of $r$-partitions. After that we recall some basic result on Fock space over quantum affine algebras and the Ariki-Lascoux-Leclerc-Thibon theory on $v$-decomposition. Throughout this section, by Lemma \ref{adjustCharge}, we assume without loss of generality that \(\boldsymbol{s} \in \widetilde{\mathcal A}^r_e(\boldsymbol{\mu})\) for every $r$-partition \(\boldsymbol{\mu}\) under consideration.

\smallskip
\subsection{Definition of adding empty runner for $r$-partitions}
\begin{dfn}\label{lamPlus} Fix an integer $\alpha$ with $0\le \alpha<e$. Let $(\blam,\bs)$ be a given pair. Consider the $e$-tuple column abacus associated with \(\tau_{e, r}\big(L_{\boldsymbol{s}}(\boldsymbol{\lambda})\big)\). Let  \(\bar{\rho}\) be a new runner such that every slot above the dashed line is occupied by a bead, whereas all slots beneath the dashed line remain vacant. Applying Lemma \ref{Uglov}, there is a uniquely determined pair $(\blam^+,\bs^+)$, such that \(\tau_{e+1, r}\big(L_{\boldsymbol{s}^+}(\boldsymbol{\lambda}^+)\big)\) is equal to the $(e+1)$-tuple column abacus formed by inserting the new runner \(\bar{\rho}\) before \(\rho_\alpha\) within the $e$-tuple column abacus associated with \(\tau_{e, r}\big(L_{\boldsymbol{s}}(\boldsymbol{\lambda})\big)\).
\end{dfn}

The construction of $L_{\bs^+}(\boldsymbol{\lambda}^+)$ from $L_{\bs}(\blam)$ corresponds to adding a column to the left of every column $j$ with \(j\in\mathbb Z\) and \(j\equiv \alpha\pmod e\) on the $(e,\bs)$-abacus \(L_{\boldsymbol{s}}(\boldsymbol{\lambda})\). If \(j\ge \alpha\), the newly added column is entirely empty; if \(j<\alpha\) (hence $j<0$), the newly added column is fully occupied by beads. Applying Lemma \ref{bead minus}, we deduce that for each $1\leq i\leq r$, $s_i^+=s_i$. Hence $\bs^+=\bs$ and $L_{\bs^+}(\blam^+)=L_{\bs}(\blam^+)$. Note that, although our notation does not reflect this, the abacus $L_{\bs}(\blam^+)$ does depend upon the choice of $\alpha$. We usually label the runners of the $(e+1)$-tuple column abacus of $\tau_{e+1,r}({L}_{\bs^+}(\blam^+))$ as $\rho^+_0,\rho^+_1,\cdots,\rho^+_{e}$ from left to right.

\begin{remark}\label{keyrem1} Note that in Definition \ref{lamPlus} we put no restrictions on the length of the partitions, and we disallow to insert the new runner \(\bar{\rho}\) to the right of the right most (i.e., $(e-1)$-th) runner $\rho_{e-1}$ in the $e$-tuple column abacus associated with \(\tau_{e, r}\big(L_{\boldsymbol{s}}(\boldsymbol{\lambda})\big)\). That says, in the special case when $r=1$, our above definition of runner inserting is slightly different with James-Mathas's runner inserting in \cite[\S2]{JM}.
\end{remark}

\begin{example}
Let $e=3$, $\bs=(2,2, 4)$ and $\blam=((2), (3,1),(1,1))$. Then $L_{\bs}(\blam)$ is as follows.
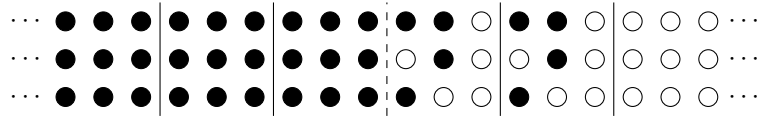
\begin{center}
\begin{tikzpicture}[scale=0.5, bb/.style={draw,circle,fill,minimum size=2.5mm,inner sep=0pt,outer sep=0pt}, wb/.style={draw,circle,fill=white,minimum size=2.5mm,inner sep=0pt,outer sep=0pt}]
	
\foreach \x in {10,-9}
\foreach \y in {2,1,0}
{
\node at (\x,\y) {$\cdots$};
}	
	\node [wb] at (9,2) {};
	\node [wb] at (8,2) {};
	\node [wb] at (7,2) {};
	\node [wb] at (6,2) {};
	\node [bb] at (5,2) {};
	\node [bb] at (4,2) {};
	\node [wb] at (3,2) {};
	\node [bb] at (2,2) {};
	\node [bb] at (1,2) {};
	\node [bb] at (0,2) {};
	\node [bb] at (-1,2) {};
	\node [bb] at (-2, 2) {};
	\node [bb] at (-3, 2) {};
	\node [bb] at (-4,2) {};
	\node [bb] at (-5,2) {};
	\node [bb] at (-6,2) {};
	\node [bb] at (-7,2) {};
	\node [bb] at (-8,2) {};

	\node [wb] at (9,1) {};
	\node [wb] at (8,1) {};
	\node [wb] at (7,1) {};
	\node [wb] at (6,1) {};
	\node [bb] at (5,1) {};
	\node [wb] at (4,1) {};
	\node [wb] at (3,1) {};
	\node [bb] at (2,1) {};
	\node [wb] at (1,1) {};
	\node [bb] at (0,1) {};
	\node [bb] at (-1,1) {};
	\node [bb] at (-2, 1) {};
	\node [bb] at (-3, 1) {};
	\node [bb] at (-4,1) {};
	\node [bb] at (-5,1) {};
	\node [bb] at (-6,1) {};
	\node [bb] at (-7,1) {};
	\node [bb] at (-8,1) {};

	\node [wb] at (9,0) {};
	\node [wb] at (8,0) {};
	\node [wb] at (7,0) {};
	\node [wb] at (6,0) {};
	\node [wb] at (5,0) {};
	\node [bb] at (4,0) {};
	\node [wb] at (3,0) {};
	\node [wb] at (2,0) {};
	\node [bb] at (1,0) {};
	\node [bb] at (0,0) {};
	\node [bb] at (-1,0) {};
	\node [bb] at (-2, 0) {};
	\node [bb] at (-3, 0) {};
	\node [bb] at (-4,0) {};
	\node [bb] at (-5,0) {};
	\node [bb] at (-6,0) {};
	\node [bb] at (-7,0) {};
	\node [bb] at (-8,0) {};
	
	\draw[](-5.5,-0.5)--node[]{}(-5.5,2.5);
	\draw[](-2.5,-0.5)--node[]{}(-2.5,2.5);
	\draw[dashed](0.5,-0.5)--node[]{}(0.5,2.5);
		\draw[](3.5,-0.5)--node[]{}(3.5,2.5);
	\draw[](6.5,-0.5)--node[]{}(6.5,2.5);
	\end{tikzpicture}
\end{center}
Let $\alpha=0$, then $L_{\bs}(\blam^+)$ is as follows.
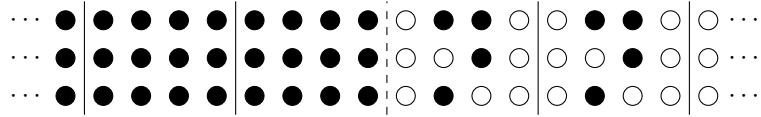
\begin{center}
\begin{tikzpicture}[scale=0.5, bb/.style={draw,circle,fill,minimum size=2.5mm,inner sep=0pt,outer sep=0pt}, wb/.style={draw,circle,fill=white,minimum size=2.5mm,inner sep=0pt,outer sep=0pt}]
	
\foreach \x in {10,-9}
\foreach \y in {2,1,0}
{
\node at (\x,\y) {$\cdots$};
}	
	\node [wb] at (9,2) {};
	\node [wb] at (8,2) {};
	\node [bb] at (7,2) {};
	\node [bb] at (6,2) {};
	\node [wb] at (5,2) {};
	\node [wb] at (4,2) {};
	\node [bb] at (3,2) {};
	\node [bb] at (2,2) {};
	\node [wb] at (1,2) {};
	\node [bb] at (0,2) {};
	\node [bb] at (-1,2) {};
	\node [bb] at (-2, 2) {};
	\node [bb] at (-3, 2) {};
	\node [bb] at (-4,2) {};
	\node [bb] at (-5,2) {};
	\node [bb] at (-6,2) {};
	\node [bb] at (-7,2) {};
	\node [bb] at (-8,2) {};

	\node [wb] at (9,1) {};
	\node [wb] at (8,1) {};
	\node [bb] at (7,1) {};
	\node [wb] at (6,1) {};
	\node [wb] at (5,1) {};
	\node [wb] at (4,1) {};
	\node [bb] at (3,1) {};
	\node [wb] at (2,1) {};
	\node [wb] at (1,1) {};
	\node [bb] at (0,1) {};
	\node [bb] at (-1,1) {};
	\node [bb] at (-2, 1) {};
	\node [bb] at (-3, 1) {};
	\node [bb] at (-4,1) {};
	\node [bb] at (-5,1) {};
	\node [bb] at (-6,1) {};
	\node [bb] at (-7,1) {};
	\node [bb] at (-8,1) {};

	\node [wb] at (9,0) {};
	\node [wb] at (8,0) {};
	\node [wb] at (7,0) {};
	\node [bb] at (6,0) {};
	\node [wb] at (5,0) {};
	\node [wb] at (4,0) {};
	\node [wb] at (3,0) {};
	\node [bb] at (2,0) {};
	\node [wb] at (1,0) {};
	\node [bb] at (0,0) {};
	\node [bb] at (-1,0) {};
	\node [bb] at (-2, 0) {};
	\node [bb] at (-3, 0) {};
	\node [bb] at (-4,0) {};
	\node [bb] at (-5,0) {};
	\node [bb] at (-6,0) {};
	\node [bb] at (-7,0) {};
	\node [bb] at (-8,0) {};
	
	\draw[](-7.5,-0.5)--node[]{}(-7.5,2.5);
	\draw[](-3.5,-0.5)--node[]{}(-3.5,2.5);
	\draw[dashed](0.5,-0.5)--node[]{}(0.5,2.5);
		\draw[](4.5,-0.5)--node[]{}(4.5,2.5);
	\draw[](8.5,-0.5)--node[]{}(8.5,2.5);
	\end{tikzpicture}
\end{center}
The $e$-tuple abaci of $\tau_{e, r}(L_{\bs}(\blam))$ and $\tau_{e+1, r}(L_{\bs}(\blam^+))$ are as follows:
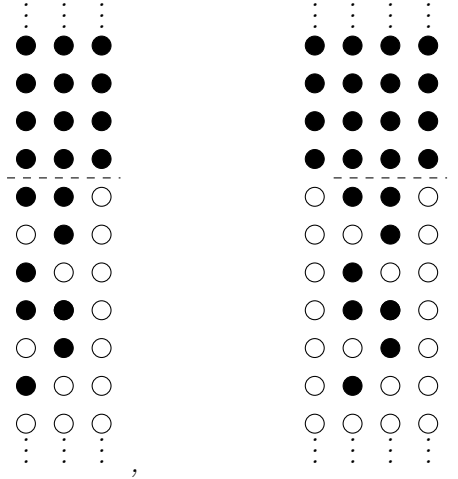
\begin{center}
\begin{tikzpicture}[scale=0.5, bb/.style={draw,circle,fill,minimum size=2.5mm,inner sep=0pt,outer sep=0pt}, wb/.style={draw,circle,fill=white,minimum size=2.5mm,inner sep=0pt,outer sep=0pt}]
\foreach \x in {-1,0,1}
\foreach \y in {5,-6.5}
{
\node at (\x,\y) {$\vdots$};
}
\node[bb] at (-1, 4){};
\node[bb] at (-1, 3){};
\node[bb] at (-1, 2){};
\node[bb] at (-1, 1){};
\node[bb] at (-1, 0){};
\node[wb] at (-1, -1){};
\node[bb] at (-1, -2){};
\node[bb] at (-1, -3){};
\node[wb] at (-1, -4){};
\node[bb] at (-1, -5){};
\node[wb] at (-1, -6){};

\node[bb] at (0, 4){};
\node[bb] at (0, 3){};
\node[bb] at (0, 2){};
\node[bb] at (0, 1){};
\node[bb] at (0, 0){};
\node[bb] at (0, -1){};
\node[wb] at (0, -2){};
\node[wb] at (0, -3){};
\node[bb] at (0, -3){};
\node[bb] at (0, -4){};
\node[wb] at (0, -5){};
\node[wb] at (0, -6){};

\node[bb] at (1, 4){};
\node[bb] at (1, 3){};
\node[bb] at (1, 2){};
\node[bb] at (1, 1){};
\node[wb] at (1, 0){};
\node[wb] at (1, -1){};
\node[wb] at (1, -2){};
\node[wb] at (1, -3){};
\node[wb] at (1, -4){};
\node[wb] at (1, -5){};
\node[wb] at (1, -6){};

\draw[dashed](-1.5,0.5)--node[]{}(1.5,0.5);
	\end{tikzpicture}
,\qquad\qquad\qquad  \begin{tikzpicture}[scale=0.5, bb/.style={draw,circle,fill,minimum size=2.5mm,inner sep=0pt,outer sep=0pt}, wb/.style={draw,circle,fill=white,minimum size=2.5mm,inner sep=0pt,outer sep=0pt}]
\foreach \x in {-2,-1,0,1}
\foreach \y in {5,-6.5}
{
\node at (\x,\y) {$\vdots$};
}
\node[bb] at (-2, 4){};
\node[bb] at (-2, 3){};
\node[bb] at (-2, 2){};
\node[bb] at (-2, 1){};
\node[wb] at (-2, 0){};
\node[wb] at (-2, -1){};
\node[wb] at (-2, -2){};
\node[wb] at (-2, -3){};
\node[wb] at (-2, -4){};
\node[wb] at (-2, -5){};
\node[wb] at (-2, -6){};

\node[bb] at (-1, 4){};
\node[bb] at (-1, 3){};
\node[bb] at (-1, 2){};
\node[bb] at (-1, 1){};
\node[bb] at (-1, 0){};
\node[wb] at (-1, -1){};
\node[bb] at (-1, -2){};
\node[bb] at (-1, -3){};
\node[wb] at (-1, -4){};
\node[bb] at (-1, -5){};
\node[wb] at (-1, -6){};

\node[bb] at (0, 4){};
\node[bb] at (0, 3){};
\node[bb] at (0, 2){};
\node[bb] at (0, 1){};
\node[bb] at (0, 0){};
\node[bb] at (0, -1){};
\node[wb] at (0, -2){};
\node[wb] at (0, -3){};
\node[bb] at (0, -3){};
\node[bb] at (0, -4){};
\node[wb] at (0, -5){};
\node[wb] at (0, -6){};

\node[bb] at (1, 4){};
\node[bb] at (1, 3){};
\node[bb] at (1, 2){};
\node[bb] at (1, 1){};
\node[wb] at (1, 0){};
\node[wb] at (1, -1){};
\node[wb] at (1, -2){};
\node[wb] at (1, -3){};
\node[wb] at (1, -4){};
\node[wb] at (1, -5){};
\node[wb] at (1, -6){};

\draw[dashed](-1.5,0.5)--node[]{}(1.5,0.5);
	\end{tikzpicture}
.
	\end{center}
In particular, $$
\tau_{e, r}(L_{\bs}(\blam))=L_8(8,7,5^2,3,2),\quad \tau_{e+1, r}(L_{\bs}(\blam^+))=L_8(14,12,9^2,6,4,1^2),\quad \blam^+=((4,1),(5,2),(3^2,1^2)).
$$
\end{example}

\begin{lemma}\label{ssPlus} Suppose $\bs\in  \widetilde A^r_{e}(\blam)$. Then $\bs \in \widetilde A^r_{e+1}(\blam^+)$.
\end{lemma}

\begin{proof} Since \(\boldsymbol{s} \in \overline{\mathcal A}^r_e\), it follows that \(\boldsymbol{s}^+ \in \overline{\mathcal A}^r_{e+1}\). By our construction of \(L_{\boldsymbol{s}}(\boldsymbol{\lambda}^+)\), all columns to the left of the dashed line on \(L_{\boldsymbol{s}}(\boldsymbol{\lambda}^+)\) are fully occupied by beads. Lemma \ref{s_jge} implies \(s_j \ge \ell\big((\boldsymbol{\lambda}^+)^{(j)}\big)\) for every integer \(1 \le j \le r\).
\end{proof}

If $B$ is a block of \(\H^{\Lambda_{\bs}}_{r,n}\) corresponding to $\alpha\in Q_n^+$, we write \(w(B)\) for the weight of $B$ as defined in \cite{F1}. It is well-known that $w(B):=2(\Lambda_{\bs},\alpha)-(\alpha,\alpha)$.

\begin{lemma}\label{block moving vector}
Let $B$ denote the block of $\H_{r,n}^{\Lam_{\bs}}$ containing $(\blam, \bs)$. Let $B_1$ denote the block of $\H_{r,n^+}^{\Lam_{\bs}}$ containing $(\blam^+, \bs)$, where $n^+:=|\lambda^+|$. Then the block moving vector of $B$ is equal to the block moving vector of $B_1$. In particular, $w(B)=w(B_1)$.
\end{lemma}

\begin{proof}
The bead moves defined on \(L_{\boldsymbol{s}}(\boldsymbol{\lambda})\) correspond to the following rule on \(\tau_{e,r}\big(L_{\boldsymbol{s}}(\boldsymbol{\lambda})\big)\): if a bead has an empty position directly above it, we may shift that bead into the vacant spot. By the properties of the additional
runner we inserted, this runner does not interfere with all such bead moves.

Then the block moving vector of $B$ is equal to the block moving vector of $B_1$. In particular, by \cite[Lemma 4.1.15]{LQ2025}, $w(B)=w(B_1)$.
\end{proof}

%\medskip
%\subsection{}
Let $\widetilde{\Sym}_e$ be the affine Weyl group of type $A_{e-1}^{(1)}$. It is generated by $\sigma_j, j\in\Z/e\Z$,
which satisfy the following relations: $$\begin{aligned}
& \sigma_i^2=1,\quad\forall\,i\in\Z/e\Z,\\
& \sigma_i\sigma_{i+1}\sigma_i=\sigma_{i+1}\sigma_i\sigma_{i+1},\quad\forall\,i\in I,\\
& \sigma_i\sigma_j=\sigma_j\sigma_i,\quad\forall\,j\neq i\pm 1.
\end{aligned}
$$
There is a faithful action of $\widetilde{\Sym}_e$ on $\Z$ defined as follows (see \cite[\S1.2.3]{F2}): $\forall\,j\in I, k\in\Z$, $$
\sigma_j\cdot k=\begin{cases}  k-1, &\text{if $k\equiv j\pmod{e}$,}\\
 k+1, &\text{if $k\equiv j-1\pmod{e}$,}\\
 k, &\text{otherwise.}\\
\end{cases}
$$
This induces an action of $\widetilde{\Sym}_e$ on the set of $r$-abaci.

\begin{lemma}\text{\rm \text{\rm (\cite[Lemma 3.2.1]{LQ2025})})}\label{3.2.1} Let $L$ be an $r$-abacus diagram (with a vertical dashed line separating positions $-1$ and $0$). Assume that for each \(1 \leq i \leq r\), both the number of beads on \(L_i\) lying to the right of the dashed line and the number of empty positions on \(L_i\) lying to the left of the dashed line are finite. Let $\bs$ be the multicharge of $L$ as defined in (\ref{Biject1}).
Then for any $j\in I$, the multicharge of $\sigma_j(L)$ is still equal to $\bs$.
\end{lemma}

Applying Lemma \ref{3.2.1}, we see that the action of $\widetilde{\Sym}_e$ on the set of $r$-abaci induces an action of $\widetilde{\Sym}_e$ on $\PP_{r,n}$. In particular, for any $j\in I$,
the $r$-abacus of $(\sigma_j(\blam),\bs)$ is obtained by interchanging columns $j-1+ke$ and $j+ke$ in the $r$-abacus of $L_\bs(\blam)$ for all $k\in \Z$.

\medskip
\subsection{Fock spaces over the quantum affine algebra $U_v(\widehat{\mathfrak{sl}}_e)$}

The simple modules of the cyclotomic Hecke algebra $\H_{r,n}^{\Lam_{\bs}}$ can be parameterized by the set of regular $r$-partitions. The $(e,\bs)$-regular $r$-partitions are generalizations of the $e$-regular partitions and can be viewed as a dual version of the so-called Kleshchev $r$-partitions. In this subsection, we shall recall the definition of $(e,\bs)$-regular $r$-partitions.

\begin{dfn} Let $i\in I$. For any two $i$-nodes $\gamma=(b,c,j)$ and $\gamma'=(b',c',j')$, we say that $\gamma$ is below $\gamma'$ (equivalently, $\gamma'$ is above $\gamma$) if either $j>j'$ or $j=j'$ and $b>b'$. In this case, we write $\gamma\succ\gamma'$. Given an $r$-partition $\blam$ and an $i$-node $\gamma$, we write $\rem_{i}{\blam}\uparrow_\gamma$ and $\add_{i}{\blam} \uparrow_\gamma$ for the number of removable $i$-nodes  and addable $i$-nodes of $\blam$ above $\gamma$, respectively.
\end{dfn}
Note this order restricts to a total order on the set of all addable and removable nodes of residue $i \in I$ of an $r$-partition.

Suppose $\blam$ is an $r$-partition and $i \in\Z/e\Z$. We define the \emph{$i$-signature of $\blam$ with respect to ``$\succ$}'' by examining all the addable and removable $i$-nodes of $\blam$ in turn from lower to higher, and writing an ``$A$'' for each addable $i$-node and an ``$R$'' for each removable $i$-node. Now construct the \emph{reduced $i$-signature with respect to $\succ$} by successively deleting all adjacent pairs ``$RA$''. If there are any ``$R$'' in the reduced $i$-signature of $\blam$, the lowest of these nodes is called the \emph{good $i$-node of $\blam$ with respect to ``$\succ$''}.

%(It may be necessary to replace it with the definition of the dual Kleshchev $r$-partition.)

\begin{definition}
We say that $\blam$ is an \emph{$(e,\bs)$-regular $r$-partition} if and only if there is a sequence
$$\blam= \blam(n), \blam(n-1), \ldots, \blam(0)= \bvarnothing$$
of $r$-partitions such that for each $k$, $[\blam(k-1)]$ is obtained from $[\blam(k)]$ by removing a good node of  $[\blam(k)]$ with respect to $\succ$.
\end{definition}

\begin{lemma}\text{\rm (\cite[Lemma 1.9]{F2})}\label{Kleequ}
Let $\blam$ be an $r$-partition, $0\le i\le e-1$. Suppose either $\blam$ has no addable $i$-nodes or $\blam$ has no removable $i$-nodes  . Then $\blam$ is an $(e,\bs)$-regular $r$-partition if and only if $\sigma_i(\blam)$ is an $(e,\bs)$-regular $r$-partition.
\end{lemma}
\begin{proof} Suppose  $\blam$ has no addable $i$-nodes. In \cite[Lemma 1.9]{F2}), Fayers proved the statement for Kleshchev $r$-partitions. But the same argument also works for $(e,\bs)$-regular $r$-partitions. This proves the lemma when $\blam$ has no addable $i$-nodes.

Now suppose $\blam$ has no removable $i$-nodes. Then $\sigma_i(\blam)$ has no addable $i$-nodes. So applying the first part of the lemma, we again deduce that $\blam$ is an $(e,\bs)$-regular $r$-partition if and only if $\sigma_i(\blam)$ is an $(e,\bs)$-regular $r$-partition.
\end{proof}

One of the central problems in the study of the cyclotomic Hecke algebra $\H_{r,n}^{\Lambda_{\bs}}$ is computing the decomposition numbers. That is, the composition multiplicity of each simple module $D^\bmu$ in each Specht module $S^\lambda$. According to the Ariki-Lascoux-Leclerc-Thibon theory, when \(\operatorname{char}K=0\), the decomposition numbers \([S^\blam:D^\bmu]\) can be obtained by evaluating the coefficient polynomial \(d^{e,\bs}_{\blam,\bmu}(v)\) at \(v=1\). This polynomial $d^{e,\bs}_{\blam,\bmu}(v)$ arises from expanding the canonical basis element \(G(\bmu)\) of the integral highest-weight module \(V(\Lambda_\bs)\) in terms of the natural basis \(\{s_{\blam}\}\) of the Fock space $\mathcal{F}(\Lambda_\bs)$.

Recall that the quantised enveloping algebra $U_v(\widehat{\mathfrak{sl}}_e)$ is a $\mathbb{Q}(v)$-algebra with standard generators $E_i, F_i$ for $i\in I$ and $v^h$ for $h\in P^\vee$, where $P^\vee$ is a free $\mathbb{Z}$-module with basis $\{h_i\mid i\in I\}\cup\{d\}$.
These generators are subjected to some well known relations which can be found in \cite[\S 4.1]{LLT96}. Here we follow the usual notation for $v$-integers, $v$-factorials and $v$-binomial coefficients:
$$[k]=\dfrac{v^{k}-v^{-k}}{v-v^{-1}}, \quad [k]!=[k][k-1] \cdots [1], \quad \left[\begin{matrix}
m\\k
\end{matrix}\right]=\dfrac{[m]!}{[m-k]![k]!}.$$
For any integer $m>0$, we write $F_i^{(m)}$ to denote the quantum divided power $F_i^m/[m]!$.
Since $U_v(\widehat{\mathfrak{sl}}_e)$ is a Hopf algebra with co-multiplication, the tensor product of two $U_v(\widehat{\mathfrak{sl}}_e)$-modules naturally becomes a $U_v(\widehat{\mathfrak{sl}}_e)$-module.
The $\mathbb{Q}$-linear ring automorphism $\overline{\phantom{o}}:\,\,U_v(\widehat{\mathfrak{sl}}_e)\to U_v(\widehat{\mathfrak{sl}}_e)$ defined by
\[\overline{E_i}:=E_i,\qquad \overline{F_i}:=F_i,\qquad \overline v^k:=v^{-k}, \qquad \overline{v^h}:=v^{-h}\]
for $i\in I$, $k\in\Z$, and $h\in P^\vee$ is called \emph{bar involution}.

\begin{definition} We define the \textbf{Fock space} \(\mathcal{F}_{\bs}=\mathcal{F}(\Lambda_\bs)\) to be the \(\mathbb{Q}(v)\)-vector space with basis \(\{s_{\blam} \mid \blam\in\mathscr{P}_{r,n},\, n\in\mathbb{N}\}\), which we refer to as the \textbf{standard basis}.
\end{definition}

The Fock space $\mathcal{F}_{\bs}$ has a structure of $U_v(\widehat{\mathfrak{sl}}_e)$-module (\cite[Proposition 2.6]{AM})\footnote{Note that the $r$-partition $\blam$ in this paper should be identified with $\blam'$ in the notations of \cite{AM}}, which can be described as follows: for any $i\in \Z/e\Z$, $h\in P^\vee$,
\begin{equa}\label{action1}\left\{\begin{aligned}
 E_is_\blam:&=\sum_{\gamma\in\rem_i(\blam)}v^{\#\rem_i(\blam)\downarrow_{\gamma}-\#\add_i(\blam)\downarrow_{\gamma}}s_{\blam\setminus\gamma},\\
 F_is_\blam:&=\sum_{\gamma\in\add_i(\blam)}v^{\#\add_i(\blam)\uparrow_{\gamma}-\#\rem_i(\blam)\uparrow_{\gamma}}s_{\blam\vee\gamma},\\
 v^h s_\blam:&=v^{\langle h,\Lambda_\bs-\sum_{i\in I}c_i(\blam)\alpha_i\rangle}s_\blam=v^{\#\add_i(\blam)-\#\rem_i(\blam)}s_\blam,
\end{aligned}\right.
\end{equa}
where $\add_i(\blam)\!\uparrow_{\gamma}$ (resp., $\add_i(\blam)\!\downarrow_{\gamma}$) denotes the set of addable $i$-nodes of $\blam$ above $\gamma$ (resp., below $\gamma$), $\rem_i(\blam)\!\uparrow_{\gamma}$ (resp., $\rem_i(\blam)\!\downarrow_{\gamma}$) denotes the set of removable $i$-nodes of $\blam$ above $\gamma$ (resp., below $\gamma$), $c_i(\blam)$ denote the set of $i$-nodes of $\blam$.

Let $m\in\N$, $\blam\in\mathscr{P}_{r,n}, \bmu\in\mathscr{P}_{r,n+m}$, We write $\blam \xrightarrow{m:i}\bmu $ to indicate that $\bmu$ is obtained from $\blam$ by adding $m$ addable $i$-nodes of $\blam$. Equivalently, this means that the abacus display for $L_\bs(\bmu)$ is obtained from that of $L_\bs(\blam)$ by moving $m$ beads from runner $\rho_{i-1}$ to runner $\rho_i$. In this case, we define the integer
\begin{equa}\label{action000}
N_i(\blam,\bmu)= \sum_{\gamma\in\bmu\setminus\blam}\Bigl(\#\add_{i}{\bmu} \uparrow_\gamma-\#\rem_{i}{\blam} \uparrow_\gamma\Bigr).\end{equa}
Now the action of $F_i^{(m)}$ is given by \begin{equa}\label{action111}
F_i^{(m)}s_\blam = \sum_{\blam\xrightarrow{m:i}\bmu}v^{N_i(\blam, \bmu)}s_\bmu. \end{equa}
If $m=1$ then we write $\blam \xrightarrow{i}\bmu$ instead of $\blam \xrightarrow{1:i}\bmu $ for simplicity.

\begin{proposition}\text{\rm (\cite[Proposition 3.2]{DellA24b})}\label{Ncomps}
Let $i\in I$. Suppose $\blam$ and $\bmu$ are $r$-partitions such that $\blam\xrightarrow{m:i}\bmu$. Then
$$
N_i( \blam, \bmu)=\sum_{\gamma\in \bmu \backslash \blam}\Bigl(\#\add_i (\mu^{(J_\gamma)} )\uparrow_\gamma-\#\rem_i (\lambda^{(J_\gamma)})\uparrow_\gamma\Bigr)+ \sum^{J_\gamma-1}_{j=1}\Bigl(\#\add_i (\mu^{(j)})-\#\rem_i (\lambda^{(j)})\Bigr),
$$
where $J_{\gamma}$ is the component of the node $\gamma$ in $\bmu$.
\end{proposition}

\begin{dfn} Let $\blam, \bmu\in \mathscr{P}_{r,n}$. We write $\blam\unrhd \bmu$ (or $\bmu\unlhd\blam$) if $$\sum_{t=1}^{s-1}|\blam^{(t)}|+\sum_{i=1}^j\blam_i^{(s)}
\geq\sum_{t=1}^{s-1}|\bmu^{(t)}|+\sum_{i=1}^j\bmu_i^{(s)}$$ for all $1\leq s\leq r$ and all $j\geq 1$.  Write $\blam\rhd\bmu$ (or $\bmu\lhd\blam$)  if
$\blam\unrhd\bmu$ and $\blam\neq\bmu$.
\end{dfn}

Let $M_{\bs}$ be the $U_v(\widehat{\mathfrak{sl}}_e)$-submodule of $\mathcal{F}_{\bs}$ generated by the \emph{empty} $r$-partition $\bvarnothing=(\varnothing, \ldots, \varnothing)$. It is well-known that $M_{\bs}$ is isomorphic to the simple highest weight $U_v(\widehat{\mathfrak{sl}})_{e}$-module $V(\Lambda_\bs)$.  This submodule $M_{\bs}$ inherits a bar involution from $U_v(\widehat{\mathfrak{sl}}_e)$: this is defined by $\overline{\bvarnothing}=\bvarnothing$ and $\overline{um} = \overline u\,\overline m$ for all $u\in U_v(\widehat{\mathfrak{sl}}_e)$ and $m\in M_{\bs}$.  The bar involution allows one to define a \emph{canonical basis} for $M_{\bs}$; this consists of vectors $G^{\bs}_e(\bmu)$, for $\bmu$ running over $(e,\bs)$-regular $r$-partitions. These canonical basis vectors are uniquely characterised by the following properties: for each $(e,\bs)$-regular $r$-partition $\bmu$ of $n$,
\begin{itemize}
\item[(1)] $\overline{G^{\bs}_e(\bm\mu)}=G^{\bs}_e(\bm\mu)$;
\item[(2)] $G^{\bs}_e(\bmu) = s_\bmu+\sum_{\substack{\blam\in\PP_{r,n},\\ \blam\lhd\bmu}}d^{e,\bs}_{\blam\bm\mu}(v)s_\blam$, where $d^{\bs}_{\blam\bm\mu}(v)\in v\mathbb{Z}[v]$ for each $\blam$.
\end{itemize}

\begin{theorem}\text{\rm (\cite{A1})}\label{A1} Let $\blam\in\PP_{r,n}$ and $\bmu$ be an $(e,\bs)$-regular $r$-partition of $n$. Then $$
[S^\blam:D^\bmu]=d^{e,\bs}_{\blam,\bmu}(1).
$$
\end{theorem}

Similarly, there is a $U_v(\widehat{\mathfrak{sl}}_{e+1})$-module structure on the Fock space $\mathcal F_\bs$. We want to compare the actions of $U_v(\widehat{\mathfrak{sl}}_e)$ and $U_v(\widehat{\mathfrak{sl}}_{e+1})$ on the Fock space $\mathcal F_\bs$. In order to distinguish between these two algebras and their modules, we use $\mathcal F_\bs^+$ to denote this Fock space viewed as $U_v(\widehat{\mathfrak{sl}}_{e+1})$-module. Then $M_\bs^+:=U_v (\widehat{\mathfrak{sl}}_{e+1})\bvarnothing$ is isomorphic to the simple highest weight $U_v(\widehat{\mathfrak{sl}})_{e+1}$-module with highest weight $\Lambda_\bs$. Given an $(e,\bs)$-regular $r$-partition $\bnu$, we shall show in Proposition \ref{PlusRegular} that $\bnu^+$ is $(e+1,\bs)$-regular. Let $G_{e+1}^\bs(\bnu^+)=s_{\bnu^+}+\sum_{\brho\lhd\bnu^+} d^{e+1,\bs}_{\brho\bnu^+}(v) s_\brho$ be the corresponding canonical basis element of $M_\bs^+ \subset \mathcal F_\bs^+$. In this paper we shall first compare $d^{e,\bs}_{\blam,\bnu}(v)$ with $d^{e+1,\bs}_{\blam^+,\bnu^+}(v)$, then use Ariki's theorem \ref{A1} to compare the two decomposition numbers $d^{e,\bs}_{\blam,\bnu}$ and $d^{e+1,\bs}_{\blam^+,\bnu^+}$.

\bigskip
\section{Comparing $v$-decomposition numbers upon adding empty runner}

In this section, we shall prove our first main result comparing decomposition numbers upon adding an empty runner.

\medskip
\subsection{$(e,\mathbf{s})$-regular $r$-partitions upon adding empty runner}
The following lemma is a useful result for the study of $(e,\bs)$-regularity of the $r$-partitions $\blam^+$.
\begin{lemma}\label{no+}
Fix $\blam\in\PP_{r,n}$ and $\bs\in\tilde{\mathcal{A}}_e^r(\blam)$. Let $\alpha$ be an integer with $0\le \alpha<e$. Let $\blam^+$ be defined using $\bs$ and $\alpha$ as in Definition \ref{lamPlus}. Then $L_\bs(\blam^+)$ has neither any removable $\alpha$-nodes, nor any addable $(\alpha+1)$-nodes.
\end{lemma}

\begin{proof} We consider the position $(f, g)$ in the $(e,\mathbf{s})$-abacus $L_\mathbf{s}(\boldsymbol{\lambda}^+)$, where $1 \le f \le r$ and $g \equiv \alpha \pmod{e+1}$. By the discussion below Definition \ref{lamPlus}, the position $(f, g)$ is occupied by a bead if and only if $g < 0$. However, in $L_\mathbf{s}(\boldsymbol{\lambda}^+)$, if $g-1 < 0$, then position $(f, g-1)$ also has a bead. Thus, in $L_\mathbf{s}(\boldsymbol{\lambda}^+)$, there exists no $(f, g)$ with $1 \le f \le r$ and $g \equiv\alpha\pmod{e+1}$ such that $(f, g-1)$ is empty and $(f, g)$ is a bead. By Lemma \ref{add and rem}, $\boldsymbol{\lambda}^+$ has no removable $\alpha$-nodes.

In $L_\mathbf{s}(\blam^+)$, if $1 \le x \le r$ and $y=\alpha+z(e+1), z\in \mathbb Z$, then the position $(x, y)$ has a bead if and only if $y < 0$. Now assume that $(x, y)$ has a bead. Then $y< 0$. Suppose further that $(x, y+1)$ has a bead too. Then by Lemma \ref{s_jge}, Lemma \ref{ssPlus} and the assumption that  \(\boldsymbol{s} \in \widetilde{\mathcal A}^r_e (\boldsymbol{\lambda})\) we deduce that $y+1=0$. This implies that $$
-z(e+1)=\alpha+1\in\{1,2,\cdots,e\},
$$
which is impossible! This proves that in $L_\mathbf{s}(\boldsymbol{\lambda}^+)$, there exists no $(x, y)$ with $1 \le x \le r$ and $y \equiv \alpha \pmod{e+1}$ such that $(x, y)$ is a bead and $(x, y+1)$ is a empty. Hence $\boldsymbol{\lambda}^+$ has no addable $(\alpha+1)$-nodes.
\end{proof}

\begin{remark}\label{keyrem2} Note that the validity of the above lemma relies on the definition of $\blam^+$ in Definition \ref{lamPlus}, see remark \ref{keyrem1}. If we allow to insert the new runner \(\bar{\rho}\) to the right of the right most (i.e., the $(e-1)$-th) runner $\rho_{e-1}$ in the $e$-tuple column abacus associated with \(\tau_{e, r}\big(L_{\boldsymbol{s}}(\boldsymbol{\lambda})\big)\) in the definition of $\blam^+$, then the lemma does not necessarily hold.
\end{remark}

\begin{definition} We label the runners of the $e$-tuple column abacus \(\tau_{e,r}(L_{\boldsymbol{s}}(\boldsymbol{\lambda}))\) as \(\rho_0,\dots,\rho_{e-1}\), ordered from left to right. For any integer $j$, we use $\col_j(\blam)$ (resp., $\col_j(\blam^+)$ ) to denote the $j$th column of the $(e,\bs)$ row abacus $L_\bs(\blam)$ of $\blam$ (resp., $L_\bs(\blam^+)$ of $\blam^+$). By the definition of Uglov map, for each $0\leq i<e$, the $i$-th runner of the $e$-tuple column abacus \(\tau_{e,r}(L_{\boldsymbol{s}}(\boldsymbol{\lambda}))\) may be obtained by stacking, from top to bottom, all the column $\col_j(\blam)$ of the $(e,\bs)$ row abacus $L_{\bs}(\boldsymbol{\lambda})$ satisfying $j\equiv i\pmod e$, in increasing order of $j$.
\end{definition}

\begin{prop}\label{PlusRegular} Fix $\bmu\in\PP_{r,n}$ and $\bs\in\tilde{\mathcal{A}}_e^r(\bmu)$. Let $\alpha$ be an integer with $0\le \alpha<e$.
Let $\bmu^+$ be defined using $\bs$ and $\alpha$ as in Definition \ref{lamPlus}. Assume $\bmu$ is $(e,\bs)$-regular.  Then $\bmu^+$ is an $(e+1,\bs)$-regular $r$-partition too.
\end{prop}
\begin{proof} We prove the lemma by induction on $n$. If $n=0$. Then $(\bvarnothing,\bs)$ is the unique $r$-partition of $0$. Since $w(\bvarnothing,\bs)=0$, by Lemma \ref{block moving vector}, $w(\bvarnothing^+, \bs)=0$.  Applying \cite[Theorem 4.1]{F1}, we see that $(\bvarnothing^+, \bs)$ lies in a simple block.
Hence $(\bvarnothing^+, \bs)$ is an $(e,\mathbf{s})$-regular $r$-partition.

Suppose that the statement holds for $n = k - 1, k \ge 1$. Now we assume $n=k$. Because $\bmu$ is an $(e,\bs)$-regular $r$-partition of $n$, there exists $\gamma\in [\bmu]$, $\gamma$ is a good $i$-node of $\bmu$ with respect to $\succ$, $\blam:=\bmu\setminus\{\gamma\}$ is also an
$(e,\bs)$-regular $r$-partition. Let $\bmu^+,\blam^+$ both be defined using $\bs$ and $\alpha$ as in Definition \ref{lamPlus}. By the induction hypothesis, $\blam^+$ is an $(e+1,\mathbf{s})$-regular $r$-partition, we want to show that $\bmu^+$ is an $(e+1,\mathbf{s})$-regular $r$-partition too.
Assume that the removable $i$-node $\gamma$ of $\bmu$ corresponds to the empty position in $(a, b-1)$ and the occupied position $(a, b)$ in $L_\bs (\bmu)$, where $b=i+ce$, $0\leq i<e+1, c\in\mathbb{Z}$. Set $b'=i+c(e+1)$.

Suppose $i<\alpha$. Then for any integer $j$ satisfying $j\equiv i\pmod e$, the columns $\col_{j-1}(\bmu)$ and $\col_j(\bmu)$ for $L_\bs(\bmu)$ are the same as the columns $\col_{j-1}(\bmu^+)$ and $\col_j(\bmu^+)$ for $L_\bs(\bmu^+)$. By definition of $\bmu^+$, the positions $(a, b'-1)$ (resp., $(a, b')$) in $L_\bs (\bmu^+)$ can be identified with the positions $(a, b-1)$ (resp., $(a, b)$) in $L_\bs (\bmu)$. Therefore, in $L_\bs (\bmu^+)$, position $(a, b'-1)$ is empty, position $(a, b')$ has a bead. This correspond to a removable $i$-node $\gamma'$ of $\bmu^+$. By Lemma \ref{add and rem}, the addable $i$-nodes and removable $i$-nodes for $\blam$ correspond exactly to the addable $i$-nodes and removable $i$-nodes for $\blam^+$ respectively. Moreover, this correspondence preserves the order ``$\succ$''.
Therefore, $\gamma'$ is  a good $i$-node of $\bmu^+$.
In $L_\mathbf{s}(\boldsymbol{\mu}^+)$, moving the bead at $(a, b')$ to the position $(a, b'-1)$ we obtain $L_\mathbf{s}(\boldsymbol{\lambda}^+)$. By the inductive hypothesis, ${\blam}^+$ is an $(e+1,\bs)$-regular $r$-partition.
Therefore $\bmu^+$ is an $(e+1,\bs)$-regular $r$-partition. Suppose $i>\alpha$. Then a similar argument analogous to the case where $i < \alpha$ proves the lemma.

Suppose $i=\alpha$. Then for any integer $j$ satisfying $j\equiv i\pmod e$, the columns $\col_{j-1}(\bmu)$ and $\col_j(\bmu)$ for $L_\bs(\bmu)$ are the same as the columns $\col_{j-1}(\bmu^+)$ and $\col_{j+1}(\bmu^+)$ for
$L_{\bs}(\bmu^+)$. By Lemma \ref{no+}, $\bmu^+$ has no addable $(i+1)$-nodes. By Lemma \ref{Kleequ}, $\bmu^+$ is $(e+1,\bs)$-regular if and only if $\sigma_{i+1}(\bmu^+)$ is  $(e+1,\bs)$-regular.
Then the columns $\col_{j-1}(\sigma_{i+1}(\bmu^+))$ and $\col_j(\sigma_{i+1}(\bmu^+))$ for $L_\bs(\sigma_{i+1}(\bmu^+))$ are the same as the columns $\col_{j-1}(\bmu^+)$ and $\col_{j+1}(\bmu^+)$ for $L_\bs(\bmu^+)$ respectively, hence are also the same as the columns $\col_{j-1}(\bmu)$ and $\col_j(\bmu)$ for $L_\bs(\bmu)$ respectively. Therefore, in $L_\bs (\sigma_{i+1}(\bmu^+))$, position $(a, b'-1)$ is empty and position $(a, b')$ has a bead. This corresponds to a removable $i$-node $\gamma''$ of $\sigma_{i+1}(\bmu^+)$. By a similar argument used in the last paragraph, we see $\gamma''$ is  a good $i$-node of $\sigma_{i+1}(\bmu^+)$. On $L_\mathbf{s}(\sigma_{i+1}(\boldsymbol{\mu}^+))$, moving the bead at $(a, b')$ to the position $(a, b'-1)$, we denote the resulting abacus by $L_\mathbf{s}(\boldsymbol{\nu})$.
So $\sigma_{i+1}(\bmu^+)$ is $(e+1,\bs)$-regular if and only if $\bnu$ is $(e+1,\bs)$-regular. By Lemma \ref{no+}, $L_\mathbf{s}(\boldsymbol{\lambda}^+)$ has no addable $(i+1)$-nodes. By definition and construction, we have $L_\bs (\sigma_{i+1} (\blam^+))=L_\bs (\bnu)$ and hence $\sigma_{i+1} (\blam^+)=\bnu$. Since ${\blam}^+$ is $(e+1,\bs)$-regular (by inductive hypothesis), $\bnu=\sigma_{i+1}(\blam^+)$ is $(e+1,\bs)$-regular (by Lemma \ref{Kleequ}). It follows that $\sigma_{i+1}(\bmu^+)$ and hence $\bmu^+$ is an $(e+1,\bs)$-regular $r$-partition.
\end{proof}

\begin{lemma} Let $\blam,\bmu\in\PP_{r,n}$, $\bs\in \overline{\mathcal A}^r_e$, and $\alpha$ be an integer with $0\le \alpha<e$. Let $\blam^+,\bmu^+$ be both defined as in Definition \ref{lamPlus} using $\bs$ and $\alpha$.
Then
$(\blam,\bs)$ and $(\bmu,\bs)$ lie in the same block of ${\H}_{r,n}^{\Lam_\bs}$ if and only if $(\blam^+,\bs)$ and $(\bmu^+,\bs)$ lie in the same block of ${\H}_{r,n^{+}}^{\Lam_\bs}$, where $n^+:=|\blam^+|$.
\end{lemma}

\begin{proof} By \cite[Corollary 2.27]{JL}, $(\blam,\bs)$ and $(\bmu,\bs)$ lie in the same block of ${\H}_{r,n}^{\Lam_\bs}$ if and only if $(\blam,\bs)$ and $(\bmu,\bs)$ have the same reduced $(e,\bs)$-core. By \cite[1,2, in Page 118, Remark 4.3]{JL}, one can do a series elementary operations (or bead moves) on \(L_{\boldsymbol{s}}(\boldsymbol{\lambda})\) to get the $r$-abacus of the reduced $(e,\bs)$-core of $(\blam,\bs)$. The bead moves defined on  correspond to the following operation on \(\tau_{e, r}(L_{\boldsymbol{s}}(\boldsymbol{\lambda}))\): if a bead has an empty position directly above it, then we may shift this bead into that empty spot. Owing to the properties of the runner we appended, this extra runner does not interfere with such bead moves. Therefore, \((\boldsymbol{\lambda}, \boldsymbol{s})\) and \((\boldsymbol{\mu}, \boldsymbol{s})\) share the same reduced $(e,\bs)$-core (see \cite[Definition 4.1.5]{LQ2025}) if and only if \((\boldsymbol{\lambda}^+, \boldsymbol{s}^+)\) and \((\boldsymbol{\mu}^+, \boldsymbol{s}^+)\) share the same reduced $(e+1,\bs)$-core.
The claim now follows from \cite[Corollary 2.27]{JL} (or Lemma \ref{block moving vector} and \cite[Lemma 4.1.11]{LQ2025}).
\end{proof}

\medskip
\subsection{Proof of the first main result}

Fix $\bs\in \overline{\mathcal A}^r_e$ and an integer $0 \le \alpha < e$. Let $\blam\in\PP_{r,n}$. Let $\blam^+$ be defined using $\bs$ and $\alpha$ as in Definition \ref{lamPlus}. Recall that the quantum affine algebra $U_v(\widehat{\mathfrak{sl}}_e)$ is generated by $E_i,F_i, i\in \Z/e\Z$ and $v^h, h\in P^\vee$, where $E_i,F_i,i\in\Z/e\Z$ are called the Chevalley generators. To avoid confusion, we use $\tilde{E}_i,\tilde{F}_i,i\in\Z/(e+1)\Z$ to denote the Chevalley generators of $U_v(\widehat{\mathfrak{sl}}_{e+1})$. Recall also that \(\mathcal{F}_\bs\) is the Fock space \(\mathcal{F}_\bs=\bigoplus_{\substack{n\in\mathbb{N}\\ \lambda\in\mathcal{P}_{r,n}}}\mathbb{Q}(v)s_\lambda,\) viewed as a left \(U_v(\widehat{\mathfrak{sl}}_e)\)-module. We then define \(\mathcal{F}_\bs^+:=\mathcal{F}_\bs\) as a \(\mathbb{Q}(v)\)-vector space, but endow \(\mathcal{F}_\bs^+\) with the structure of a left \(U_v(\widehat{\mathfrak{sl}}_{e+1})\)-module.

\begin{definition}\label{tildeFi1} Fix $\bs\in \overline{\mathcal A}^r_e$ and an integer $0 \le \alpha < e$. Let $\blam\in\PP_{r,n}$. Let $\blam^+$ be defined using $\bs$ and $\alpha$ as in Definition \ref{lamPlus}. For each $0\leq i<e$ and $a\in \Z_{\geq 1}$, let ${^\alpha}\!F_i: \mathcal F_\bs^+ \to \mathcal F_\bs^+$ be the $\Q(v)$-linear map which is uniquely determined by: 
$$
{^\alpha}\!F^{(a)}_i :=
\begin{cases}
\tilde{F}^{(a)}_i  & \text{if}\ 0\le i<\alpha; \\
\tilde{F}^{(a)}_{i+1} \tilde{F}^{(a)}_i  &\text{if}\ i=\alpha; \\
\tilde{F}^{(a)}_{i+1}  & \text{if}\ \alpha<i<e.
\end{cases}.
$$
Let $\varTheta: \mathcal F_\bs \to \mathcal F_\bs^+$ be the $\Q(v)$-linear map which is uniquely determined by $\varTheta(s_\blam)=s_{\blam^+}$ for any $\blam\in\PP_{r,n}$.
\end{definition}

\begin{lemma}\label{expansion} Fix $\bs\in \overline{\mathcal A}^r_e$ and an integer $0 \le \alpha < e$. Let $\blam\in\PP_{r,n}$. Let $\blam^+$ be defined using $\bs$ and $\alpha$ as in Definition \ref{lamPlus}. Let $i:=\alpha$. Then $s_\bmu$ occurs in the expansion of ${^\alpha}\!F_i^{(a)}\circ \varTheta(s_\blam)=\tilde{F}^{(a)}_{i+1}\tilde{F}^{(a)}_i s_{\blam^+}$ only if $\bmu=\bnu^+$ for some $\bnu\in\PP_{r,n}$
\end{lemma}

\begin{proof}
We show that $s_\bmu$ occurs in the expansion of ${^\alpha}\!F_i^{(a)}\circ \varTheta(\blam)=\tilde{F}^{(a)}_{i+1}\tilde{F}^{(a)}_i s_{\blam^+}$ only if $\bmu=\bnu^+$ for some $\bnu\in\PP_{r,n}$. In fact, assume that $s_\bmu$ occurs in the expansion of $\tilde{F}^{(a)}_{i+1}\tilde{F}^{(a)}_i s_{\blam^+}$. Then there exists an $r$-partition $\boldsymbol{\rho}$ such that $s_{\boldsymbol{\rho}}$ occurs in $\tilde{F}^{(a)}_i s_{\blam^+}$ and $s_\bmu$ occurs in $\tilde{F}^{(a)}_{i+1}s_{\boldsymbol{\rho}}$. In particular, we have $\blam^+\xrightarrow{a:i} {\boldsymbol{\rho}} \xrightarrow{a:i+1} \bmu$. We write $\brho=\blam^+\cup\{\gamma^+_1, \cdots, \gamma^+_a\}$.
Suppose that for each $1\leq d\leq a$, adding $\gamma^+_d$ to $\blam^+$ corresponds to moving the bead at $(a_d, b_d-1)$ to $(a_d, b_d)$ on $L_\bs(\blam^+)$, where $b_d=i+c_d(e+1)$. By Lemma \ref{no+}, $\blam^+$ has no addable $(i+1)$-nodes. By assumption, ${\boldsymbol{\rho}} \xrightarrow{a:i+1} \bmu$. It follows that $\blam^+ \cup \{\gamma^+_1, \cdots, \gamma^+_a\}$ has precisely $a$ addable $(i+1)$-nodes and $\blam^+\cup \{\gamma^+_d\}$ has a unique addable $(i+1)$-node. We denote this addable $(i+1)$-node by $\tilde \gamma_d$. It corresponds to the configuration on $L_\bs(\blam^+\cup \{\gamma^+_d\})$ where $(a_d, b_d)$ is occupied by a bead and $(a_d, b_d+1)$ is empty. This implies that on $L_\bs(\blam^+)$ for each $1\leq d\leq a$, the position $(a_d, b_d-1)$ is occupied by a bead while $(a_d, b_d+1)$ is empty.
Consequently, on $L_\bs(\blam)$, the position $(a_d, i-1+c_de)$ is occupied by a bead and $(a_d, i+c_de)$ is empty, which corresponds to an addable $i$-node $\gamma_d$ of $\blam$. Setting $\boldsymbol{\eta}=\blam\cup \{\gamma_1, \cdots, \gamma_d\}$, we see that $\bmu=\boldsymbol{\eta}^+$. This completes the proof of our claim.
\end{proof}

\begin{prop}\label{diagram commute} Fix $\bs\in \overline{\mathcal A}^r_e$ and an integer $0 \le \alpha < e$. Let $\blam\in\PP_{r,n}$. Let $\blam^+$ be defined using $\bs$ and $\alpha$ as in Definition \ref{lamPlus}.
For any integers $a\ge 1$ and $0\le i<e$, the following diagram commutes: \[ \xymatrix{
\mathcal F_\bs \ar[rr]^{F^{(a)}_i} \ar[dd]_{\varTheta} &  &\mathcal F_\bs \ar[dd]^{\varTheta} \\
&   & \\
\mathcal F_\bs^+ \ar[rr]^{{^\alpha}\!F^{(a)}_i} &  & \mathcal F_\bs^+
}.  \]
\end{prop}

\begin{proof}
Let $\blam \in \mathscr P_{r, n}$. We fix an $r$-partition $\bnu$ such that $\blam\xrightarrow{a:i} \bnu$, where $0\le i<e$. We write $[\blam] \cup \{\gamma_1, \cdots, \gamma_a\}=\bnu$, where each $\gamma_d$ ($1\le d\le a$) is an addable $i$-node of $\blam$. For any $1\le d\le a$, $\gamma_d$ corresponds to the configuration on $L_\bs(\blam)$ where $(a_d, b_d-1)$ is occupied by a bead and $(a_d, b_d)$ is empty, with $b_d=i+c_de$.

For any $1\le d\le a$, define
$$N_i(\blam, \bnu)[d]=\Bigl(\#\add_i (\nu^{(J_{\gamma_d})} )\uparrow_{\gamma_d}-\#\rem_i (\lambda^{(J_{\gamma_d})})\uparrow_{\gamma_d}\Bigr)+ \sum^{J_{\gamma_d}-1}_{j=1}\Bigl(\#\add_i (\nu^{(j)})-\#\rem_i (\lambda^{(j)})\Bigr),$$
$N_i(\blam, \bnu)=\sum_{d=1}^a N_i(\blam, \bnu)[d].$ There are three possibilities:

{\it Case 1.} $i<\alpha$. In this case for any integer $j$ satisfying $j\equiv i\pmod e$, the columns $\col_{j-1}(\blam)$ and $\col_j(\blam)$ for $L_\bs(\blam)$ are the same as the columns $\col^+_{j-1}(\blam^+)$ and $\col^+_j(\blam^+)$ for $L_\bs(\blam^+)$.
By Lemma \ref{add and rem}, the addable and removable $i$-nodes for $\blam$ correspond exactly to the addable and removable $i$-nodes for $\blam^+$.
Let $[\blam^+]\cup \{\gamma^+_1, \cdots, \gamma^+_a\}=\bnu^+$. For $1\le d\le a$, each $\gamma_d^+$ corresponds to the configuration on $L_\bs(\blam^+)$ where $(a_d, i-1+c_d(e+1))$ is occupied by a bead and $(a_d, i+c_d(e+1))$ is empty.
For any $1\le d\le a$, define
$$N_i(\blam^+, \bnu^+)[d]=\Bigl(\#\add_i ({\nu^+}^{(J_{\gamma^+_d})} )\uparrow_{\gamma^+_d}-\#\rem_i ({\lambda^+}^{(J_{\gamma^+_d})})\uparrow_{\gamma^+_d}\Bigr)+ \sum^{J_{\gamma^+_d}-1}_{j=1}\Bigl(\#\add_i ({\nu^+}^{(j)})-\#\rem_i ({\lambda^+}^{(j)})\Bigr),$$
$N_i(\blam^+, \bnu^+)=\sum_{d=1}^a N_i(\blam^+, \bnu^+)[d].$
Hence, for all $1\le d\le a$, $N_i(\blam^+, \bnu^+)[d]=N_i(\blam, \bnu)[d]$. That says, the coefficient of $s_{\bnu^+}$ in $\varTheta\circ F_i^{(a)}(\blam)$ is the same as the coefficient of $s_{\bnu^+}$ in ${^\alpha}\!F_i^{(a)}\circ \varTheta(\blam)$.

{\it Case 2.} By the same argument used in Case 1 we show that for all $1\le d\le a$, $N_i(\blam^+, \bnu^+)[d]=N_i(\blam, \bnu)[d]$ in this case. That says, the coefficient of $s_{\bnu^+}$ in $\varTheta\circ F_i^{(a)}(\blam)$ is the same as the coefficient of $s_{\bnu^+}$ in ${^\alpha}\!F_i^{(a)}\circ \varTheta(\blam)$.

{\it Case 3.} $i=\alpha$. In this case for any integer $j$ satisfying $j\equiv i\pmod e$, the columns $\col_{j-1}(\blam)$ and $\col_j(\blam)$ for $L_\bs(\blam)$ are equal to the columns $\col^+_{j-1}(\blam^+)$ and $\col^+_{j+1}(\blam^+)$ for $L_\bs(\blam^+)$ respectively.
Let $1\le a\le r$, $b\equiv\alpha \pmod{e+1}$. By Definition \ref{lamPlus}, if $b<0$ then the position $(a, b)$ in $L_\bs(\blam^+)$ has a bead; while if $b\ge 0$ then the position $(a, b)$ in $L_\bs(\blam^+)$ is empty.

Any addable $i$-node of $\blam$ corresponds to a pair $(x,y)$, where $1 \le x \le r$ and $y=i+ze$, such that the position $(x, y-1)$ in the $r$-abacus $L_\bs(\blam)$ contains a bead and the position $(x, y)$ is empty.  Since the position $(x, y)$ is empty, it follows that $y \ge 0$ and hence $z\geq 0$, which implies $y':= i+z(e+1) \ge 0$. Note that the columns $\col_{y-1}(\blam)$ and $\col_{y}(\blam)$ for $L_\bs(\blam)$ become the columns
$\col_{y'-1}(\blam^+)$ and $\col_{y'+1}(\blam^+)$ for $L_\bs(\blam^+)$ after inserting the column $\overline{\rho}$. It follows that the position $(x, y'-1)$ $L_\bs(\blam^+)$ is occupied by a bead while the position $(x, y')$ in $L_\bs(\blam^+)$ is empty, which corresponds to an addable $i$-node $\gamma^+$ of $\blam^+$. Similarly, one shows that any removable $i$-node $\eta$ of $\blam$ corresponds to a removable $(i+1)$-node $\eta^+$ of $\blam^+$. Thus
the addable $i$-nodes and removable $i$-nodes of $\blam$ are in one-to-one correspondence with two subsets $J_{\blam^+,i}^{\rm{add}}, J_{\blam^+,i}^{\rm{rem}}$  of the addable $i$-nodes and removable $(i+1)$-nodes of $\blam^+$ respectively. Moreover, any addable $i$-node $\tilde{\gamma}$ of $\blam^+$ does not lie in $J_{\blam^+,i}^{\rm{add}}$ if and only if $\tilde{\gamma}$ corresponds to a position $(x,y')$ in $L_\bs(\blam^+)$ such that both positions $(x, y-1)$ is empty and $(x, y)$ in $L_\bs(\blam)$ are occupied by beads; similarly, any removable $(i+1)$-node $\tilde{\eta}$ of $\blam^+$ does not lie in $J_{\blam^+,i}^{\rm{rem}}$ if and only if $\tilde{\eta}$ corresponds to a position  $(x,y')$ is empty in $L_\bs(\blam^+)$ such that both positions $(x, y-1)$ and $(x, y)$ in $L_\bs(\blam)$ are occupied by beads.

Let $1\leq d\leq a$. We set $$
A:=\add_i(\blam)\uparrow_{\gamma_d},\,\,\, B:=\rem_{i}(\bnu)\uparrow_{\gamma_d}=\rem_{i}(\blam) \uparrow_{\gamma_d}. $$
Then by the discussion in the last paragraph, we have \begin{equa}\label{biiplus1}B=\{\beta\in J_{\blam^+,i+1}^{\rm{rem}}\mid\beta\succ{\gamma_d}\}.\end{equa} Let $\bxi$ be an $r$-partition such that $\blam^+\xrightarrow{a:i} \bxi \xrightarrow{a:i+1} \bnu^+$.
From the proof of Lemma {expansion}, we know that $\blam^+\cup \{\gamma^+_1, \cdots, \gamma^+_a\}=\bxi$, $\bxi\cup \{\tilde \gamma^+_1, \cdots, \tilde \gamma^+_a\}=\bnu^+$, where $\tilde \gamma_d$ is the unique addable $(i+1)$-node on $\blam^+\cup \{\gamma^+_a\}$ for $1\le d\le a$.
For any $1\le d\le a$, define
$$N_i(\blam^+, \bxi)[d]=\Bigl(\#\add_i (\xi^{(J_{\gamma^+_d})} )\uparrow_{\gamma^+_d}-\#\rem_i ({\lambda^+}^{(J_{\gamma^+_d})})\uparrow_{\gamma^+_d}\Bigr)+ \sum^{J_{\gamma^+_d}-1}_{j=1}\Bigl(\#\add_i (\xi^{(j)})-\#\rem_i ({\lambda^+}^{(j)})\Bigr),$$
$N_i(\blam^+, \bxi)=\sum_{d=1}^a N_i(\blam^+, \bxi)[d].$
For any $1\le d\le a$, define
$$N_i(\bxi, \bnu^+)[d]=\Bigl(\#\add_i ({\nu^+}^{(J_{\tilde \gamma_d})} )\uparrow_{\tilde \gamma_d}-\#\rem_i ({\xi}^{(J_{\tilde \gamma_d})})\uparrow_{\tilde \gamma_d}\Bigr)+ \sum^{J_{\tilde \gamma_d}-1}_{j=1}\Bigl(\#\add_i ({\nu^+}^{(j)})-\#\rem_i ({\xi}^{(j)})\Bigr),$$
$N_i(\bxi, \bnu^+)=\sum_{d=1}^a N_i(\bxi, \bnu^+)[d].$
By Lemma \ref{no+}, $\blam^+$ has no removable $i$-nodes.
Therefore, if we let $l$ denote the number of positions \((x,y)\) which lie to the right and below the position \((a_d,b_d)\) in \(L_\bs(\blam)\) and such that both \((x,y-1)\) and \((x,y)\) are occupied by beads in \(L_\bs(\blam)\), then $N_i(\blam^+, \bxi)[d]=\#A+l$.
Recall that $\blam^{+}\cup\{\gamma^+_d\}$ has a unique addable $(i+1)$-node $\tilde \gamma_d$.
Moreover, each removable $i$-node of $\blam$ naturally corresponds to a removable $(i+1)$-node of $\blam^+$ and hence corresponds to a removable $(i+1)$-node of $\bxi$. Therefore, by (\ref{biiplus1}), $N_{i+1}(\bxi, \bnu^+)[d]=-(\#B+l)$. Consequently, $N_i(\blam^+,\bxi)[d]+N_{i+1} (\bxi,\bnu^+)[d]=\#A-\#B=N_i(\blam, \bnu)[d]$.
Hence the coefficient of $s_{\bnu^+}$ in $\varTheta\circ F_i^{(a)}(\blam)$ is the same as the coefficient of $s_{\bnu^+}$ in ${^\alpha}\!F_i^{(a)}\circ \varTheta(\blam)=\tilde{F}^{(a)}_{i+1}\tilde{F}^{(a)}_i \blam^+$.
This completes the proof of the proposition.
\end{proof}

\begin{dfn}\text{\rm (\cite[10.2]{J})} Let ``$\prec$'' be the lexicographic order on $\PP_{r,n}$. That says, for any $\blam=(\lambda^{(1)},\cdots,\lambda^{(r)})$, $\bmu=(\mu^{(1)},\cdots,\mu^{(r)})\in\PP_{r,n}$, $\blam\prec\bmu$ if there exists $j\in \{1,\ldots,r\}$ and $k\in \mathbb{Z}_{>0}$ such that $\lambda^{(a)}=\mu^{(a)}$ for all $a\in \{1,\ldots,j-1\}$, and $\lambda^{(j)}_m=\mu^{(j)}_m$ for all $m\in \{1,\ldots,k-1\}$ and $\lambda^{(j)}_k<\mu^{(j)}_k$.
\end{dfn}
It is clear that $\blam \lhd \bmu$ implies $\blam\prec\bmu$. Recall that $\bs=(s_1,\cdots,s_r)\in\Z^r$. Our definition of $(e,\bs)$-regular $r$-partition depends only on $s_j+e\Z\in\Z/e\Z$, $j=1,2,\cdots,r$. Since $e>1$, without loss of generality, we can assume that $s_j-s_{j-1}\ge n-1+e$ for $j=2, \cdots, r$. Note also that the $(e,\bs)$-regular $r$-partitions in this paper are the same as the Kleshchev $r$-partitions used in \cite{J}.
Let $\bmu$ be an $(e,\bs)$-regular $r$-partition of $n$. Applying \cite[\S10]{J}, we can find a sequence of residues (called the associated staggered sequence of residues): $$
\underbrace{{i}_1,\ldots,{i}_{1}}_{a_1\text{ times}}  \underbrace{{i}_2,\ldots,{i}_{2}}_{a_2\text{ times}} \ldots, \underbrace{{i}_m,\ldots,{i}_{m}}_{a_m\text{ times}}, $$
where ${i}_{j}\in \mathbb{Z}/e\mathbb{Z}$ and $a_j\in \mathbb{Z}_{>0}$ for all $j\in \{1,\ldots,m\}$ and where we assume that ${i}_s \neq {i}_{s+1}$ for all $s\in \{1,\ldots,m-1\}$, such that $$
F_{{i}_1}^{(a_1)}\ldots F_{{i}_m}^{(a_m)}\bvarnothing =\bmu + \sum_{\blam\prec\bmu}  c_{\blam \bmu} (v) \blam,$$
where $c_{\blam\bmu} (v)\in\Z[v,v^{-1}]$ for each pair $(\blam,\bmu)$.

\begin{lemma}\label{4.9}
Fix $0\leq \alpha<e$ and the multicharge $\bs\in\tilde{\mathcal A}^r_e(\blam)\cap\tilde{\mathcal A}^r_e(\bmu)$, where $\blam,\bmu$ are two $r$-partitions. We have $\blam \prec \bmu$ if and only if $\blam^+ \prec \bmu^+$.
\end{lemma}

\begin{proof} $\blam \prec \bmu$ if and only if there exist $j\in \{1,\ldots,r\}$ and $k\in \mathbb{Z}_{>0}$ such that $\lambda^{(a)}=\mu^{(a)}$ for all $a\in \{1,\ldots,j-1\}$, $\lambda^{(j)}_m=\mu^{(j)}_m$ for all $m\in \{1,\ldots,k-1\}$, and $\lambda^{(j)}_k<\mu^{(j)}_k$.

Recall $\CIRCLE_j^i(\blam, \bs)$ denotes the $j$-th bead on the abacus $L_{s_i}(\blam^{(i)})$, counted from right to left. By definition, $\CIRCLE_j^i(\blam, \bs)$ is located in column $s_i+\lambda^{(i)}_j-j$ on the $i$-th runner $L_{s_i}(\blam^{(i)})$ of the $r$-abacus $L_\bs(\blam)$. Hence, $\blam \prec \bmu$ if and only if $L_{s_a}(\blam^{(a)})=L_{s_a}(\bmu^{(a)})$ for all $a\in \{1, \ldots, j-1\}$, $\CIRCLE^j_m$ of $L_{s_j}(\lambda^{(j)})$ and that of $L_{s_j}(\mu^{(j)})$ are in the same column for all $m\in \{1, \ldots, k-1\}$, and $\CIRCLE^j_k$ of $L_{s_j}(\lambda^{(j)})$ lies strictly to the left of $\CIRCLE^j_k$ of $L_{s_j}(\mu^{(j)})$.

Since $\lambda^{(j)}_k<\mu^{(j)}_k$, we have $\mu^{(j)}_k>0$. Thus, there exists an empty position to the left of $\CIRCLE^j_k$ in $L_{s_j}(\mu^{(j)})$, which was placed it to the right of the dashed line. Consequently, no bead will be inserted to its right during the transformation from $L_\bs(\bmu)$ to $L_\bs(\bmu^+)$.

Note that the beads inserted from $L_\bs(\blam)$ to $L_\bs(\blam^+)$ and from $L_\bs(\bmu)$ to $L_\bs(\bmu^+)$ are in the same columns. Therefore, $L_{s_a}({\lambda^+}^{(a)})=L_{s_a}({\mu^+}^{(a)})$ for all $a\in \{1, \ldots, j-1\}$, $\CIRCLE^j_m$ of $L_{s_j}({\lambda^+}^{(j)})$ and that of $L_{s_j}({\mu^+}^{(j)})$ are in the same column for all $m\in \{1, \ldots, k-1\}$, and $\CIRCLE^j_k$ of $L_{s_j}({\lambda^+}^{(j)})$ lies to the left of $\CIRCLE^j_k$ of $L_{s_j}({\mu^+}^{(j)})$. It follows that $\blam \prec \bmu$ if and only if $\blam^+ \prec \bmu^+$.
\end{proof}

\begin{proposition}\label{KeyCanonicalBases}
Let $\bmu\in\PP_{r,n}$, $\bs\in \tilde{\mathcal A}^r_e(\bmu)$. We assume $\bs\in \tilde{\mathcal A}^r_e(\blam)$ whenever $(\blam, \bs)$ and $(\bmu, \bs)$ lie in the same block. Fix an integer $0 \le \alpha < e$. Let $\bmu^+$ be defined using $\bs$ and $\alpha$ as in Definition \ref{lamPlus}.
Suppose that $\bmu$ is an $(e,\bs)$-regular. Then $G^{\bs}_{e+1}(\bmu^+)=\varTheta (G^{\bs}_e(\bmu))$.
\end{proposition}

\begin{proof} We have $$
G^\bs_e (\bmu) = s_\bmu+\sum_{\bmu\rhd\blam \in \PP_{r,n}} d^{e,\bs}_{\blam \bmu}(v) s_\blam, $$
where $d^\bs_{\blam \bmu} (v) \in v\Z[v]$ for each pair $(\blam,\bmu)$. Since $\bmu\rhd\blam$ implies $\bmu\succ\blam$, it follows that $$
G^\bs_e (\bmu) = s_\bmu+\sum_{\bmu\succ\blam \in \PP_{r,n}} d^{e,\bs}_{\blam \bmu}(v) s_\blam. $$
Then  $$
\varTheta (G^\bs_e(\bmu)) = s_{\bmu^+}+\sum_{\bmu\succ\blam \in \PP_{r,n}} d^{e,\bs}_{\blam \bmu}(v) s_{\blam^+} . $$
Applying Lemma \ref{4.9},  we get $$
\varTheta (G^\bs_e(\bmu)) = s_{\bmu^+}+\sum_{\substack{\blam\in\PP_{r,n},\\ \bmu^+\succ\blam^+}} d^{e,\bs}_{\blam \bmu}(v) s_{\blam^+} . $$

By Proposition \ref{PlusRegular}, $\bmu^+$ is $(e,\bs)$-regular. Therefore, $$
G^\bs_{e+1}(\bmu^+) = s_{\bmu^+} + \sum_{\bmu^+\succ\brho\in \PP_{r,n^+}} d^{e+1,\bs}_{\brho\bmu^+}(v)s_{\brho} ,
$$
Note that $G^\bs_{e+1} (\bmu^+)$ is uniquely determined by the properties that $$
\overline{G^\bs_{e+1} (\bmu^+)} = G^\bs_{e+1}(\bmu^+),\quad G^\bs_{e+1} (\bmu^+)\equiv s_{\bmu^+}\pmod{\sum_{\bmu^+\succ\brho}v\Z[v]s_{\brho}}, $$
where $n^+:=|\bmu^+|$. Therefore, to prove $\varTheta(G_e^\bs (\bmu)) = G^\bs_{e+1}(\bmu^+)$, it suffices to show that $\overline{\varTheta(G_e^\bs (\bmu))}=\varTheta(G_e^\bs (\bmu))$.

Following \cite[10.2]{J}, there is a sequence of residues: $$
\underbrace{{i}_1,\ldots,{i}_{1}}_{a_1\text{ times}}  \underbrace{{i}_2,\ldots,{i}_{2}}_{a_2\text{ times}} \ldots, \underbrace{{i}_m,\ldots,{i}_{m}}_{a_m\text{ times}}, $$
where ${i}_{j}\in \mathbb{Z}/e\mathbb{Z}$ and $a_j\in \mathbb{Z}_{>0}$ for all $j\in \{1,\ldots,m\}$ and where we assume that ${i}_s \neq {i}_{s+1}$ for all $s\in \{1,\ldots,m-1\}$, such that $$
F_{{i}_1}^{(a_1)}\ldots F_{{i}_m}^{(a_m)}\bvarnothing =\bmu + \sum_{\blam\prec\bmu}  c_{\blam \bmu} (v) \blam,$$
where $c_{\blam\bmu} (v)\in\Z[v,v^{-1}]$ for each pair $(\blam,\bmu)$.

Since $\overline{F_i} = F_i$ for each $0\leq i<e$, it follows that $A_\bmu = F_{{i}_1}^{(a_1)} \dots F_{{i}_m}^{(a_m)} \bvarnothing$ is bar-invariant.
Since $A_\bmu=s_\bmu + \sum_{\bmu\succ\blam} c_{\blam \bmu}(v) s_\blam$, where $c_{\blam \bmu} \in \mathbb Z[v, v^{-1}]$ for each pair $(\blam,\bmu)$. By an induction on ``$\prec$'', we deduce that there exist unique polynomials $\alpha_{\bnu \bmu}(v,v^{-1})\in \mathbb Z[v,v^{-1}]$ such that $G^\bs_ e(\bmu)=A_\bmu-\sum_{\bmu\succ\bnu} \alpha_{\bnu \bmu} (v,v^{-1}) G^\bs_e (\bnu)$, where the summation runs over $(e,\bs)$-regular $r$-partitions $\bnu$ satisfying $\bmu\succ\bnu$. Moreover, $\overline{\alpha_{\bnu \bmu} (v,v^{-1})}=\alpha_{\bnu \bmu} (v,v^{-1})$ for each pair $(\bmu,\bnu)$.

We define $A^+_\bmu := \varTheta(A_\bmu)$. Then $A^+_\bmu = s_{\bmu^+} + \sum_{\bmu\succ\blam} a_{\blam \bmu}(v,v^{-1}) s_{\blam^+}$, where $a_{\blam \bmu} \in \mathbb Z[v, v^{-1}]$ for each pair $(\blam,\bmu)$.
By Proposition \ref{diagram commute}, $A^+_\bmu={{^\alpha}\! F}_{{i}_1}^{(a_1)}\ldots {{^\alpha}\!F}_{{i}_m}^{(a_m)}\varTheta(\bvarnothing)$.
Assume $(\varTheta(\bvarnothing), \bs)$ belong to the block $B_1$. Since $w(\bvarnothing, \bs)=0$, by lemma \ref{block moving vector}, we see $w(\varTheta(\bvarnothing), \bs)=0$. Applying \cite[Theorem 4.1]{F1}, $B_1$ is simple block. Hence $G^\bs_{e+1}(\varTheta(\bvarnothing))=\varTheta (\bvarnothing)$. In particular, $\varTheta(\bvarnothing)$ is bar invariant. It follows that $A^+_\bmu$ is bar invariant.

We proceed by induction on the lexicographical order ``$\prec$''. If there exists no $(e,\bs)$-regular $r$-partition $\bxi$ such that $\bnu \succ \bxi$, then $\varTheta(G^\bs_e(\bnu)) = A^+_\mu$ is bar-invariant. Assume that $\varTheta(G^\bs_e(\bnu))$ is bar-invariant for all $(e,\bs)$-regular $r$-partitions $\bnu$ such that $\bmu\succ\bnu$; then $\varTheta(G^\bs_e(\bmu)) = A_\bmu^+ - \sum_{\bmu\succ\bnu} \alpha_{\bnu\bmu}(v,v^{-1}) \varTheta(G^\bs_e(\bnu))$ is also bar-invariant. Therefore, $\varTheta(G^\bs_e(\bmu))$ is bar invariant. Consequently, we show that $\varTheta (G^\bs_e(\bmu))=G^\bs_{e+1}(\bmu^+)$.
\end{proof}

\medskip
\noindent
{\textbf{Proof of Theorem A}: } This follows from Propositions \ref{PlusRegular}, \ref{KeyCanonicalBases} and Theorem \ref{A1}.\hfill\qed
\medskip

\begin{remark} Note that in general $\varnothing^{+}\neq\varnothing$. For example, assume $s=5, r=1, e=3, \alpha=0$. we have that $$
\varnothing^{+}=(2,2,1,1,1)\neq\varnothing.
$$
\end{remark}

\begin{example} Suppose that $e=4, \bs=(4,6), r=2, n=4, \alpha=0$, $\blam=((1^2),(1^2)), \bmu=((2),(2))$. Then we have that $$
\blam^+=((3,2,1^2),(3^2,1^4)),\,\,\quad\bmu^+=((4,1^3),(4,2,1^4)), $$
and $$
\bigl[\hat{S}^{((3,2,1^2),(3^2,1^4)))}:\hat{D}^{((4,1^3),(4,2,1^4))}\bigr]=\bigl[S^{((1^2),(1^2))}:D^{((2),(2))}\bigr]=1.
$$
\end{example}

\bigskip
\section{Theorem of adding full runner for $r$-partitions}

The purpose of this section is to generalize Fayers's full runner removal theorem \cite{Fay08} to the setting of the cyclotomic Hecke algebra $\H_{r,n}$. Throughout this section, we fix an integer \(\alpha\) satisfying \(0\le \alpha<e\). Let $\blam\in\PP_{r,n}$. We assume $\bs \in \overline{\mathcal A}^r_e$ but we do not require $\bs \in \widetilde{\mathcal A}^r_e(\blam)$.

\medskip
\subsection{Definition of adding full runner for $r$-partitions}

\begin{dfn}\label{Plusk} Let $k\in\N$ and $\blam\in\PP_{r,n}$. Consider the $e$-tuple column abacus associated with \(\tau_{e, r}\big(L_{\boldsymbol{s}}(\boldsymbol{\lambda})\big)\). Let  \(\tilde{\rho}\) be a new runner such that every slot above the dashed line and the first $kr$ bead-positions below the dashed line are occupied by beads, whereas all the remaining positions on \(\tilde{\rho}\) are empty. Applying Lemma \ref{Uglov}, there is a uniquely determined pair $(\blam^{+k},\bs^{+k})$, such that \(\tau_{e+1, r}\big(L_{\boldsymbol{s}^{+k}}(\boldsymbol{\lambda}^{+k})\big)\) is equal to the $(e+1)$-tuple column abacus formed by inserting the new runner \(\tilde{\rho}\) before \(\rho_\alpha\) within the $e$-tuple column abacus associated with \(\tau_{e, r}\big(L_{\boldsymbol{s}}(\boldsymbol{\lambda})\big)\).
\end{dfn}

This corresponds to inserting a column to the left of each column \(\col_j(\blam)\) of \(L_{\boldsymbol{s}}(\boldsymbol{\lambda})\) for all indices $j$ with \(j\equiv \alpha\pmod e\). The inserted column is completely empty when \(j\ge \alpha+ke\), and fully populated with beads when \(j\le \alpha+(k-1)e\). Although our notation does not reflect,  the abacus $L_{\bs^{+k}}(\blam^{+k})$ does depend upon the choice of $\alpha$. We usually label the runners of the $(e+1)$-tuple column abacus of $\tau_{e+1,r}({L}_{\bs^{+k}}(\blam^{+k}))$ as $\tilde{\rho}_0,\tilde{\rho}_1,\cdots,\tilde{\rho}_{e}$ from left to right.

The condition \(\boldsymbol{s} \in \widetilde{\mathcal A}^r_e(\boldsymbol{\lambda})\) is not required when inserting full columns. If we impose this condition, the insertion of an empty runner may be regarded as the case \(k=0\).

\begin{example}
Let $e=3, k=2$, $\bs=(1, 1, 2)$ and $\blam=((4,1), \varnothing, \varnothing)$. Then $L_{\bs}(\blam)$ is as follows.
\begin{center}
\begin{tikzpicture}[scale=0.5, bb/.style={draw,circle,fill,minimum size=2.5mm,inner sep=0pt,outer sep=0pt}, wb/.style={draw,circle,fill=white,minimum size=2.5mm,inner sep=0pt,outer sep=0pt}]
	
\foreach \x in {10,-9}
\foreach \y in {2,1,0}
{
\node at (\x,\y) {$\cdots$};
}	
	\node [wb] at (9,2) {};
	\node [wb] at (8,2) {};
	\node [wb] at (7,2) {};
	\node [wb] at (6,2) {};
	\node [wb] at (5,2) {};
	\node [wb] at (4,2) {};
	\node [wb] at (3,2) {};
	\node [bb] at (2,2) {};
	\node [bb] at (1,2) {};
	\node [bb] at (0,2) {};
	\node [bb] at (-1,2) {};
	\node [bb] at (-2, 2) {};
	\node [bb] at (-3, 2) {};
	\node [bb] at (-4,2) {};
	\node [bb] at (-5,2) {};
	\node [bb] at (-6,2) {};
	\node [bb] at (-7,2) {};
	\node [bb] at (-8,2) {};

	\node [wb] at (9,1) {};
	\node [wb] at (8,1) {};
	\node [wb] at (7,1) {};
	\node [wb] at (6,1) {};
	\node [wb] at (5,1) {};
	\node [wb] at (4,1) {};
	\node [wb] at (3,1) {};
	\node [wb] at (2,1) {};
	\node [bb] at (1,1) {};
	\node [bb] at (0,1) {};
	\node [bb] at (-1,1) {};
	\node [bb] at (-2, 1) {};
	\node [bb] at (-3, 1) {};
	\node [bb] at (-4,1) {};
	\node [bb] at (-5,1) {};
	\node [bb] at (-6,1) {};
	\node [bb] at (-7,1) {};
	\node [bb] at (-8,1) {};

	\node [wb] at (9,0) {};
	\node [wb] at (8,0) {};
	\node [wb] at (7,0) {};
	\node [wb] at (6,0) {};
	\node [bb] at (5,0) {};
	\node [wb] at (4,0) {};
	\node [wb] at (3,0) {};
	\node [wb] at (2,0) {};
	\node [bb] at (1,0) {};
	\node [wb] at (0,0) {};
	\node [bb] at (-1,0) {};
	\node [bb] at (-2, 0) {};
	\node [bb] at (-3, 0) {};
	\node [bb] at (-4,0) {};
	\node [bb] at (-5,0) {};
	\node [bb] at (-6,0) {};
	\node [bb] at (-7,0) {};
	\node [bb] at (-8,0) {};
	
	\draw[](-5.5,-0.5)--node[]{}(-5.5,2.5);
	\draw[](-2.5,-0.5)--node[]{}(-2.5,2.5);
	\draw[dashed](0.5,-0.5)--node[]{}(0.5,2.5);
		\draw[](3.5,-0.5)--node[]{}(3.5,2.5);
	\draw[](6.5,-0.5)--node[]{}(6.5,2.5);
	\end{tikzpicture}.
\end{center}
Let $\alpha=0$, then $L_{\bs^{+2}}(\blam^{+2})$ is as follows.
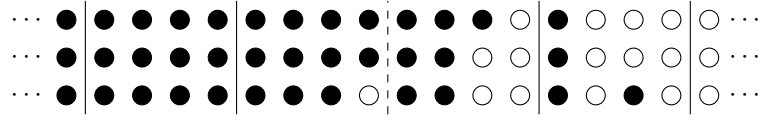
\begin{center}
\begin{tikzpicture}[scale=0.5, bb/.style={draw,circle,fill,minimum size=2.5mm,inner sep=0pt,outer sep=0pt}, wb/.style={draw,circle,fill=white,minimum size=2.5mm,inner sep=0pt,outer sep=0pt}]
	
\foreach \x in {10,-9}
\foreach \y in {2,1,0}
{
\node at (\x,\y) {$\cdots$};
}	
	\node [wb] at (9,2) {};
	\node [wb] at (8,2) {};
	\node [wb] at (7,2) {};
	\node [wb] at (6,2) {};
	\node [bb] at (5,2) {};
	\node [wb] at (4,2) {};
	\node [bb] at (3,2) {};
	\node [bb] at (2,2) {};
	\node [bb] at (1,2) {};
	\node [bb] at (0,2) {};
	\node [bb] at (-1,2) {};
	\node [bb] at (-2, 2) {};
	\node [bb] at (-3, 2) {};
	\node [bb] at (-4,2) {};
	\node [bb] at (-5,2) {};
	\node [bb] at (-6,2) {};
	\node [bb] at (-7,2) {};
	\node [bb] at (-8,2) {};

	\node [wb] at (9,1) {};
	\node [wb] at (8,1) {};
	\node [wb] at (7,1) {};
	\node [wb] at (6,1) {};
	\node [bb] at (5,1) {};
	\node [wb] at (4,1) {};
	\node [wb] at (3,1) {};
	\node [bb] at (2,1) {};
	\node [bb] at (1,1) {};
	\node [bb] at (0,1) {};
	\node [bb] at (-1,1) {};
	\node [bb] at (-2, 1) {};
	\node [bb] at (-3, 1) {};
	\node [bb] at (-4,1) {};
	\node [bb] at (-5,1) {};
	\node [bb] at (-6,1) {};
	\node [bb] at (-7,1) {};
	\node [bb] at (-8,1) {};

	\node [wb] at (9,0) {};
	\node [wb] at (8,0) {};
	\node [bb] at (7,0) {};
	\node [wb] at (6,0) {};
	\node [bb] at (5,0) {};
	\node [wb] at (4,0) {};
	\node [wb] at (3,0) {};
	\node [bb] at (2,0) {};
	\node [bb] at (1,0) {};
	\node [wb] at (0,0) {};
	\node [bb] at (-1,0) {};
	\node [bb] at (-2, 0) {};
	\node [bb] at (-3, 0) {};
	\node [bb] at (-4,0) {};
	\node [bb] at (-5,0) {};
	\node [bb] at (-6,0) {};
	\node [bb] at (-7,0) {};
	\node [bb] at (-8,0) {};
	
	\draw[](-7.5,-0.5)--node[]{}(-7.5,2.5);
	\draw[](-3.5,-0.5)--node[]{}(-3.5,2.5);
	\draw[dashed](0.5,-0.5)--node[]{}(0.5,2.5);
		\draw[](4.5,-0.5)--node[]{}(4.5,2.5);
	\draw[](8.5,-0.5)--node[]{}(8.5,2.5);
	\end{tikzpicture}.
\end{center}
The $e$-tuple abacus of $\tau_{e, r}(L_{\bs}(\blam))$ is
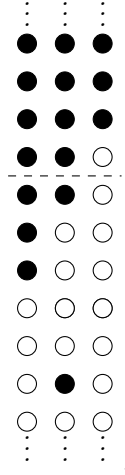
\begin{center}
\begin{tikzpicture}[scale=0.5, bb/.style={draw,circle,fill,minimum size=2.5mm,inner sep=0pt,outer sep=0pt}, wb/.style={draw,circle,fill=white,minimum size=2.5mm,inner sep=0pt,outer sep=0pt}]
\foreach \x in {-1,0,1}
\foreach \y in {5,-6.5}
{
\node at (\x,\y) {$\vdots$};
}
\node[bb] at (-1, 4){};
\node[bb] at (-1, 3){};
\node[bb] at (-1, 2){};
\node[bb] at (-1, 1){};
\node[bb] at (-1, 0){};
\node[bb] at (-1, -1){};
\node[bb] at (-1, -2){};
\node[wb] at (-1, -3){};
\node[wb] at (-1, -4){};
\node[wb] at (-1, -5){};
\node[wb] at (-1, -6){};

\node[bb] at (0, 4){};
\node[bb] at (0, 3){};
\node[bb] at (0, 2){};
\node[bb] at (0, 1){};
\node[bb] at (0, 0){};
\node[wb] at (0, -1){};
\node[wb] at (0, -2){};
\node[wb] at (0, -3){};
\node[wb] at (0, -3){};
\node[wb] at (0, -4){};
\node[bb] at (0, -5){};
\node[wb] at (0, -6){};

\node[bb] at (1, 4){};
\node[bb] at (1, 3){};
\node[bb] at (1, 2){};
\node[wb] at (1, 1){};
\node[wb] at (1, 0){};
\node[wb] at (1, -1){};
\node[wb] at (1, -2){};
\node[wb] at (1, -3){};
\node[wb] at (1, -3){};
\node[wb] at (1, -4){};
\node[wb] at (1, -5){};
\node[wb] at (1, -6){};

\draw[dashed](-1.5,0.5)--node[]{}(1.5,0.5);
	\end{tikzpicture}.
	\end{center}
The $e$-tuple abacus of $\tau_{e+1, r}(L_{\bs}(\blam^{+2}))$ is	
	
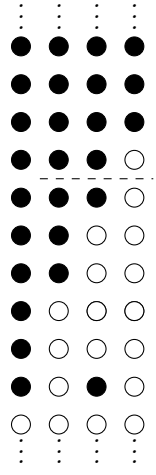
\begin{center}
\begin{tikzpicture}[scale=0.5, bb/.style={draw,circle,fill,minimum size=2.5mm,inner sep=0pt,outer sep=0pt}, wb/.style={draw,circle,fill=white,minimum size=2.5mm,inner sep=0pt,outer sep=0pt}]
\foreach \x in {-2,-1,0,1}
\foreach \y in {5,-6.5}
{
\node at (\x,\y) {$\vdots$};
}
\node[bb] at (-2, 4){};
\node[bb] at (-2, 3){};
\node[bb] at (-2, 2){};
\node[bb] at (-2, 1){};
\node[bb] at (-2, 0){};
\node[bb] at (-2, -1){};
\node[bb] at (-2, -2){};
\node[bb] at (-2, -3){};
\node[bb] at (-2, -4){};
\node[bb] at (-2, -5){};
\node[wb] at (-2, -6){};

\node[bb] at (-1, 4){};
\node[bb] at (-1, 3){};
\node[bb] at (-1, 2){};
\node[bb] at (-1, 1){};
\node[bb] at (-1, 0){};
\node[bb] at (-1, -1){};
\node[bb] at (-1, -2){};
\node[wb] at (-1, -3){};
\node[wb] at (-1, -4){};
\node[wb] at (-1, -5){};
\node[wb] at (-1, -6){};

\node[bb] at (0, 4){};
\node[bb] at (0, 3){};
\node[bb] at (0, 2){};
\node[bb] at (0, 1){};
\node[bb] at (0, 0){};
\node[wb] at (0, -1){};
\node[wb] at (0, -2){};
\node[wb] at (0, -3){};
\node[wb] at (0, -3){};
\node[wb] at (0, -4){};
\node[bb] at (0, -5){};
\node[wb] at (0, -6){};

\node[bb] at (1, 4){};
\node[bb] at (1, 3){};
\node[bb] at (1, 2){};
\node[wb] at (1, 1){};
\node[wb] at (1, 0){};
\node[wb] at (1, -1){};
\node[wb] at (1, -2){};
\node[wb] at (1, -3){};
\node[wb] at (1, -3){};
\node[wb] at (1, -4){};
\node[wb] at (1, -5){};
\node[wb] at (1, -6){};

\draw[dashed](-1.5,0.5)--node[]{}(1.5,0.5);
	\end{tikzpicture}.
	\end{center}
\end{example}

\begin{lemma} Let $k\in\N$. Fix $\bs \in \overline{\mathcal A}^r_e$ and an integer \(\alpha\) satisfying \(0\le \alpha<e\). Let $\blam\in\PP_{r,n}$. Let $(\blam^{+k},\bs^{+k})$ be the unique pair defined using $\alpha$ and $k$ as in Definition \ref{Plusk}, where  $\bs^{+k}=(s_1^{+k},\cdots,s_r^{+k})\in\Z^r$. Then $s^{+k}_i=s_i+k$ for each $1\leq i\leq r$. Moreover, $\bs^{+k} \in \overline{\mathcal A}^r_{e+1}$.
\end{lemma}

\begin{proof}
By Lemma \ref{bead minus}, for all $1\le i\le r$, assume that the number of beads in $L_{s_i}(\lambda^{(i)})$ on the right side of the dashed vertical line (between position $-1$ and $0$) is $n_{i1}$ and that number of empty positions on the left side of the dashed vertical line is $n_{i_2}$. Note that the number of beads in $L_{s_i^{+k}} ((\lambda)^+)^{(i)})$ on the right side of the dashed vertical line (between position $-1$ and $0$) is  $n_{i1}+k$ and that number of empty positions on the left side of the dashed vertical line is also $n_{i_2}$. Therefore $s_i^{+k}=n_{i_1}+k-n_{i_2}=s_i+k$. Moreover, since $\bs \in \overline{\mathcal A}^r_e$, $s_i^{+k}=s_i+k$, so $\bs \in \overline{\mathcal A}^r_{e+1}$.
\end{proof}

If we require $\bs \in \widetilde{\mathcal A}^r_e(\blam)$, then $\bs^{+k} \in \widetilde{\mathcal A}^r_e (\blam^{+k})$.

\begin{lemma}\label{block moving vector1}
Let $(\blam, \bs) \in B$, $(\blam^{+k}, \bs^{+k}) \in B_1$. Then the block moving vector of $B$ is equal to the block moving vector of $B_1$. In particular, $w(B)=w(B_1)$.
\end{lemma}

\begin{proof} This follows from the same argument used in the proof of Lemma \ref{block moving vector}.
%The bead moves defined on \(L_{\boldsymbol{s}}(\boldsymbol{\lambda})\) correspond to the following rule on \(\tau_{e,r}(L_{\boldsymbol{s}}(\boldsymbol{\lambda}))\): if a bead has an empty position directly above it, we may shift the bead into that empty spot. By the properties of the additional runner we inserted, this runner does not interfere with such bead moves.
%Then the block moving vector of $B$ is equal to the block moving vector of $B_1$. In particular, by \cite[Lemma 4.1.15]{LQ2025} $w(B)=w(B_1)$.
\end{proof}

\begin{dfn}\label{rightLeft} On the $e$-tuple column abacus of a partition $\lam$, we say that the position $(k,i)$ at the $k$th row of the $i$-th runner is immediately to the right of the position $(l,j)$ at the $l$th row of the $j$-th runner if
either $k=l, 1\leq i=j+1<e$ or $j=e-1, i=0, k=l+r$. In this case, we also  say that the position $(l,j)$ is immediately to the left of the position $(k,l)$.
\end{dfn}

Using Definition \ref{Uglovmap} and the discussion in the paragraph immediately below it, we see that any bead in the $r$-abacus $L_\bs(\blam)$ has a bead immediately to its right (resp., to its left) if and only if
the image of this bead under the Uglov map $\tau_{e,r}$ has a  bead immediately to its right (resp., to its left) in the $e$-tuple column abacus $\tau_{e,r}( L_\bs(\blam))$ in the sense of Definition \ref{rightLeft}.

\begin{lemma}\label{right}
Let $k\in\N$. Fix $\bs\in \overline{\mathcal A}^r_e$ and an integer \(\alpha\) satisfying \(0\le \alpha<e\). Let $\blam\in\PP_{r,n}$. If $(k-1)e \ge s_r+  n$, then on \(\tau_{e+1,r}(L_{\boldsymbol{s}^{+k}}(\boldsymbol{\lambda}^{+k}))\), every bead on \(\tilde{\rho}_{\alpha-1}\) has another bead immediately to its right, and every bead on \(\tilde{\rho}_{\alpha+1}\) has another bead immediately to its left.
\end{lemma}

\begin{proof}
 Since \(\boldsymbol{s} \in \overline{\mathcal A}^r_e\), we have \(s_1 \le \cdots \le s_r\). On the $r$-abacus of \((\blam,\boldsymbol{s})\), the rightmost bead appear in column $a$ only if \(a\leq s_r +\tilde n - 1\).

If $s_r+ \tilde n-1\le (k-1)e-1$, beads in the $e$-tuple column abacus of $\tau_{e,r}(L_{\boldsymbol{s}}(\boldsymbol{\lambda}))$ appear in row $b$ only if $b\leq (k-1)r-1$, where row $0$ is immediately below the dashed line and row indices increase downward. Notice that the newly added column contains beads in all rows from $0$ to $kr-1$.

Assume $1\leq\alpha\leq e-1$. On the $(e+1)$-tuple column abacus of $\tau_{e+1,r}(L_{\boldsymbol{s}^{+k}}(\boldsymbol{\lambda}^{+k}))$, if the position $(i,\alpha-1)$ in column $\alpha-1$ contains a bead, then $(i,\alpha)$ must contain a bead. Similarly, if $(k,\alpha+1)$ in column $\alpha+1$ contains a bead, then $(k,\alpha)$ must also contain a bead.

Assume $\alpha=0$. On the $(e+1)$-tuple column abacus of $\tau_{e+1,r}(L_{\boldsymbol{s}^{+k}}(\boldsymbol{\lambda}^{+k}))$, if the position $(i,e)$ in column $\alpha-1$ contains a bead, then the position $(i+r, 0)$ in column $\alpha$ contains a bead; if the position $(i, 1)$ in column $\alpha+1$ contains a bead, then the position $(i, 0)$ in column $\alpha$ contains a bead.

Assume $\alpha=e-1$. On the $(e+1)$-tuple column abacus of $\tau_{e+1,r}(L_{\boldsymbol{s}^{+k}}(\boldsymbol{\lambda}^{+k}))$, if the position $(i,e-2)$ in column $\alpha-1$ contains a bead, then the position $(i,e-1)$ in column $\alpha$ contains a bead; if the position $(i,e)$ in column $\alpha+1$ contains a bead, then the position $(i-r, e-1)$ in column $\alpha$ contains a bead.
\end{proof}

\begin{lemma}\label{no+k}
Fix $k\in\N, \bs\in \overline{\mathcal A}^r_e$ and an integer \(\alpha\) satisfying \(0\le \alpha<e\). Let $\bmu\in\PP_{r,n}$, If \((k-1)e \ge s_r +n\), then \(L_{\boldsymbol{s}^{+k}}(\boldsymbol{\lambda}^{+k})\) admits no addable \(\alpha\)-nodes and no removable \((\alpha+1)\)-nodes.
\end{lemma}

\begin{proof}
This is a direct consequence of Lemma \ref{right}.
\end{proof}

\medskip
\subsection{Proof of the second main result}

\begin{prop}\label{PlusRegular2}
Let $k\in\N, \bs\in \overline{\mathcal A}^r_e$ and an integer an integer \(\alpha\) satisfying \(0\le \alpha<e\).  Let $\bmu\in\PP_{r,n}$. If $(k-1)e \ge s_r+ n$ and $\bmu$ is an $(e,\bs)$-regular $r$-partition, then $\bmu^{+k}$ is an $(e+1,\bs^{+k})$-regular $r$-partition.
\end{prop}
\begin{proof}
We prove this by induction on $n$. If $n=0$. Then $(\bvarnothing,\bs)$ is the unique $r$-partition of $0$. Since $w(\bvarnothing,\bs)=0$, by Lemma \ref{block moving vector1}, $w(\bvarnothing^{+k},\bs^{+k})=0$.  Applying \cite[Theorem 4.1]{F1}, we see that $(\bvarnothing^{+k},\bs^{+k})$ lies in a simple block. Hence $(\bvarnothing^{+k}, \bs^{+k})$ is an $(e+1,\mathbf{s}^{+k})$-regular $r$-partition.

Suppose that the statement holds for $n = l - 1, l \ge 1$. Now we assume $n=l$. Because $(\bmu, \bs)$ is an $(e,\bs)$-regular $r$-partition, there exists $\gamma\in [\bmu]$, $\gamma$ is a good $i$-node of $\bmu$ with respect to $\succ$, $[\blam]$ is obtained from $[\bmu]$ by removing $\gamma$. $\blam$ is also an $(e,\bs)$-regular $r$-partition. Let $\bmu^{+k},\blam^{+k}$ both be defined using $\bs$ and $\alpha$ as in Definition \ref{Plusk}.
Assume that the removable $i$-node $\gamma$ of $\bmu$ corresponds to the empty position in $(a, b-1)$ and the occupied position $(a, b)$ in $L_\bs (\bmu)$, where $b=i+ce$, $0\leq i<e+1, c\in\mathbb{Z}$. Set $b'=i+c(e+1)$.

Suppose $i<\alpha$. Then for any integer $j$ satisfying $j\equiv i\pmod e$, the columns $\col_{j-1}(\bmu)$ and $\col_j(\bmu)$ for $L_\bs(\bmu)$ are the same as the columns $\col_{j-1}(\bmu^{+k})$ and $\col_j(\bmu^{+k})$ for $L_{\bs^{+k}}(\bmu^{+k})$. By definition of $\bmu^{+k}$, the positions $(a, b'-1)$ (resp., $(a, b')$) in $L_{\bs^{+k}}(\bmu^{+k})$ can be identified with the positions $(a, b-1)$ (resp., $(a, b)$) in $L_\bs (\bmu)$. Therefore, in $L_{\bs^{+k}} (\bmu^{+k})$, position $(a, b'-1)$ is empty, position $(a, b')$ has a bead. This correspond to a removable $i$-node $\gamma'$ of $\bmu^{+k}$. By Lemma \ref{add and rem}, the addable $i$-nodes and removable $i$-nodes for $\blam$ correspond exactly to the addable $(i+1)$-nodes and removable $(i+1)$-nodes for $\blam^{+k}$ respectively. Moreover, this correspondence preserves the order ``$\succ$''.
Therefore, $\gamma'$ is  a good $i$-node of $\bmu^{+k}$.
In $L_{\bs^{+k}}(\boldsymbol{\mu}^{+k})$, moving the bead at $(a, b')$ to the position $(a, b'-1)$ we obtain $L_{\bs^{+k}}(\boldsymbol{\lambda}^{+k})$. By the inductive hypothesis, ${\blam}^{+k}$ is an $(e+1,\bs^{+k})$-regular $r$-partition. Therefore $\bmu^{+k}$ is an $(e+1,\bs+k)$-regular $r$-partition. Suppose $i>\alpha$. Then a similar argument analogous to the case where $i < \alpha$ proves the lemma.

Suppose $i=\alpha$. Then for any integer $j$ satisfying $j\equiv i\pmod e$, the columns $\col_{j-1}(\bmu)$ and $\col_j(\bmu)$ for $L_\bs(\bmu)$ are the same as the columns $\col_{j-1}(\bmu^{+k})$ and $\col_{j+1}(\bmu^{+k})$ for $L_{\bs^{+k}}(\bmu^{+k})$. By Lemma \ref{no+k}, $\bmu^{+k}$ has no removable $(i+1)$-nodes. By Lemma \ref{Kleequ}, $\bmu^{+k}$ is $(e+1,\bs^{+k})$-regular if and only if $\sigma_{i+1}(\bmu^{+k})$ is  $(e+1,\bs^{+k})$-regular.
Then the columns $\col_{j-1}(\sigma_{i+1}(\bmu^{+k}))$ and $\col_j(\sigma_{i+1}(\bmu^{+k}))$ for $L_{\bs^{+k}}(\sigma_{i+1}(\bmu^{+k}))$ are the same as the columns $\col_{j-1}(\bmu^{+k})$ and $\col_{j+1}(\bmu^{+k})$ for $L_{\bs^{+k}}(\bmu^{+k})$ respectively, hence are also the same as the columns $\col_{j-1}(\bmu)$ and $\col_j(\bmu)$ for $L_\bs(\bmu)$ respectively. Therefore, in $L_{\bs^{+k}}(\sigma_{i+1}(\bmu^{+k}))$, position $(a, b'-1)$ is empty and position $(a, b')$ has a bead. This corresponds to a removable $i$-node $\gamma''$ of $\sigma_{i+1}(\bmu^{+k})$. By a similar argument used in the last paragraph, we see $\gamma''$ is  a good $i$-node of $\sigma_{i+1}(\bmu^{+k})$. On $L_{\bs^{+k}}(\sigma_{i+1}(\boldsymbol{\mu}^{+k}))$, moving the bead at $(a, b')$ to the position $(a, b'-1)$, we denote the resulting abacus by $L_{\bs^{+k}}(\boldsymbol{\nu})$.
So $\sigma_{i+1}(\bmu^{+k})$ is $(e+1,\bs^{+k})$-regular if and only if $\bnu$ is $(e+1,\bs^{+k})$-regular. By Lemma \ref{no+k}, $L_{\bs^{+k}}(\boldsymbol{\lambda}^{+k})$ has no removable $(i+1)$-nodes. By definition and construction, we have $L_{\bs^{+k}}(\sigma_{i+1} (\blam^{+k}))=L_{\bs^{+k}}(\bnu)$ and hence $\sigma_{i+1} (\blam^{+k})=\bnu$. Since ${\blam}^{+k}$ is $(e+1,\bs^{+k})$-regular (by inductive hypothesis), $\bnu=\sigma_{i+1}(\blam^{+k})$ is $(e+1,\bs^{+k})$-regular (by Lemma \ref{Kleequ}). It follows that $\sigma_{i+1}(\bmu^{+k})$ and hence $\bmu^{+k}$ is an $(e+1,\bs^{+k})$-regular $r$-partition.
\end{proof}

Fix $k\in\N, \bs\in \overline{\mathcal A}^r_e$ and an integer $0 \le \alpha < e$. Let $\blam\in\PP_{r,n}$. Let $\blam^{+k}$ be defined using $\bs$ and $\alpha$ as in Definition \ref{Plusk}. Recall that \(\mathcal{F}_\bs\) is the Fock space \(\mathcal{F}_\bs=\bigoplus_{\substack{n\in\mathbb{N}\\ \lambda\in\mathcal{P}_{r,n}}}\mathbb{Q}(v)s_\lambda,\) viewed as a left \(U_v(\widehat{\mathfrak{sl}}_e)\)-module. We then define \(\mathcal{F}_\bs^{+k}:=\mathcal{F}_\bs\) as a \(\mathbb{Q}(v)\)-vector space, but endow \(\mathcal{F}_\bs^{+k}\) with the structure of a left \(U_v(\widehat{\mathfrak{sl}}_{e+1})\)-module.

\begin{definition}\label{tildeFi} Fix $k\in\N, \bs\in \overline{\mathcal A}^r_e$ and an integer $0 \le \alpha < e$. Let $\blam\in\PP_{r,n}$. Let $\blam^{+k}$ be defined using $\bs$ and $\alpha$ as in Definition \ref{Plusk}. For each $0\leq i<e$ and $a\in \Z_{\geq 1}$, let ${^\alpha}\!\tilde{F}_i: \mathcal F_\bs^{+k} \to \mathcal F_{\bs}^{+k}$ be the $\Q(v)$-linear map which is uniquely determined by:
$$
{^\alpha}\! \tilde F^{(a)}_i :=
\begin{cases}
\tilde{F}^{(a)}_i  & \text{if}\ 0\le i<\alpha; \\
\tilde{F}^{(a)}_{i} \tilde{F}^{(a)}_{i+1}  &\text{if}\ i=\alpha; \\
\tilde{F}^{(a)}_{i+1}  & \text{if}\ \alpha<i<e.
\end{cases}.
$$
Let $\varTheta^{(k)}: \mathcal F_\bs \to \mathcal F_\bs^{+k}$ be the $\Q(v)$-linear map which is uniquely determined by $\varTheta(s_\blam)=s_{\blam^{+k}}$ for any $\blam\in\PP_{r,n}$.
\end{definition}
Note the difference between the definitions of ${^\alpha}\!{F}_i$ and ${^\alpha}\!\tilde{F}_i$  (Definitions \ref{tildeFi1} and \ref{tildeFi}) in the case when $i=\alpha$.

\begin{lemma}\label{expansion+}
Fix $\bs\in \overline{\mathcal A}^r_e$ and an integer $0 \le \alpha < e$.
Let $\blam\in\PP_{r,n}$.
Let $\blam^{+k}$ be defined using $\bs$, $\alpha$ and $k$ as in Definition \ref{Plusk}.
Let $i:=\alpha$.
If \((k-1)e \ge s_r +n\), then $s_\bmu$ occurs in the expansion of ${^\alpha}\!\tilde F_i^{(a)}\circ \varTheta(s_\blam)=\tilde{F}^{(a)}_i\tilde{F}^{(a)}_{i+1} s_{\blam^{+k}}$ only if $\bmu=\bnu^{+k}$ for some $\bnu\in\PP_{r,n}$
\end{lemma}

\begin{proof}
We show that $s_\bmu$ occurs in the expansion of ${^\alpha}\!\tilde F_i^{(a)}\circ \varTheta(\blam)=\tilde{F}^{(a)}_i\tilde{F}^{(a)}_{i+1} s_{\blam^{+k}}$ only if $\bmu=\bnu^{+k}$ for some $\bnu\in\PP_{r,n}$. In fact, assume that $s_\bmu$ occurs in the expansion of $\tilde{F}^{(a)}_i\tilde{F}^{(a)}_{i+1} s_{\blam^{+k}}$.
Then there exists an $r$-partition $\boldsymbol{\rho}$ such that $s_{\boldsymbol{\rho}}$ occurs in $\tilde{F}^{(a)}_{i+1} s_{\blam^{+k}}$ and $s_\bmu$ occurs in $\tilde{F}^{(a)}_{i}s_{\boldsymbol{\rho}}$.
In particular, we have $\blam^{+k}\xrightarrow{a:i+1} {\boldsymbol{\rho}} \xrightarrow{a:i} \bmu$.
We write $\brho=\blam^{+k}\cup\{\gamma^+_1, \cdots, \gamma^+_a\}$.
Suppose that for each $1\leq d\leq a$, adding $\gamma^+_d$ to $\blam^{+k}$ corresponds to moving the bead at $(a_d, b_d)$ to $(a_d, b_d+1)$ on $L_\bs(\blam^{+k})$, where $b_d=i+c_d(e+1)$.
By Lemma \ref{no+k}, $\blam^{+k}$ has no addable $(i+1)$-nodes. By assumption, ${\boldsymbol{\rho}} \xrightarrow{a:i} \bmu$.
It follows that $\blam^{+k} \cup \{\gamma^+_1, \cdots, \gamma^+_a\}$ has precisely $a$ addable $i$-nodes and $\blam^{+k}\cup \{\gamma^+_d\}$ has a unique addable $i$-node. We denote this addable $i$-node by $\tilde \gamma_d$. It corresponds to the configuration on $L_\bs(\blam^{+k}\cup \{\gamma^+_d\})$ where $(a_d, b_d-1)$ is occupied by a bead and $(a_d, b_d)$ is empty. This implies that on $L_\bs(\blam^{+k})$ for each $1\leq d\leq a$, the position $(a_d, b_d-1)$ is occupied by a bead while $(a_d, b_d+1)$ is empty.
Consequently, on $L_\bs(\blam)$, the position $(a_d, i-1+c_de)$ is occupied by a bead and $(a_d, i+c_de)$ is empty, which corresponds to an addable $i$-node $\gamma_d$ of $\blam$. Setting $\boldsymbol{\eta}=\blam\cup \{\gamma_1, \cdots, \gamma_d\}$, we see that $\bmu=\boldsymbol{\eta}^{+k}$. This completes the proof of our claim.
\end{proof}

\begin{prop}\label{diagram commute2}
Let $k\in\N$ and $(\bmu,\bs)$ be a given pair, where $\bmu\in\PP_{r,n}, \bs\in \overline{\mathcal A}^r_e$. Fix an integer \(\alpha\) satisfying \(0\le \alpha<e\). Suppose $(k-1)e\ge s_r+ n$. Then for any $a\in\Z_{\geq 1}$ and $0\le i<e$, the following diagram commutes:
\[
\xymatrix{
\mathcal{F}_{\bs} \ar[rr]^{F^{(a)}_i} \ar[dd]_{\varTheta^{(k)}} & & \mathcal{F}_{\bs} \ar[dd]^{\varTheta^{(k)}} \\
& & \\
\mathcal{F}_{\bs}^{+k} \ar[rr]^{{{^\alpha}\!\tilde{F}}^{(a)}_i} & & \mathcal{F}_{\bs}^{+k}
}.
\]
\end{prop}

\begin{proof}

By Lemma \ref{no+}, $\blam^+$ has no removable $i$-nodes.
Therefore, if we let $l$ denote the number of positions \((x,y)\) which lie to the right and below the position \((a_d,b_d)\) in \(L_\bs(\blam)\) and such that both \((x,y-1)\) and \((x,y)\) are occupied by beads in \(L_\bs(\blam)\), then $N_i(\blam^+, \bxi)[d]=\#A+l$.
Recall that $\blam^{+}\cup\{\gamma^+_d\}$ has a unique addable $(i+1)$-node $\tilde \gamma_d$.
Moreover, each removable $i$-node of $\blam$ naturally corresponds to a removable $(i+1)$-node of $\blam^+$ and hence corresponds to a removable $(i+1)$-node of $\bxi$. Therefore, by (\ref{biiplus1}), $N_{i+1}(\bxi, \bnu^+)[d]=-(\#B+l)$. Consequently, $N_i(\blam^+,\bxi)[d]+N_{i+1} (\bxi,\bnu^+)[d]=\#A-\#B=N_i(\blam, \bnu)[d]$.
Hence the coefficient of $s_{\bnu^+}$ in $\varTheta\circ F_i^{(a)}(\blam)$ is the same as the coefficient of $s_{\bnu^+}$ in ${^\alpha}\!F_i^{(a)}\circ \varTheta(\blam)=\tilde{F}^{(a)}_{i+1}\tilde{F}^{(a)}_i \blam^+$.
This completes the proof of the proposition.
\end{proof}

\begin{proof}
Let $\blam \in \mathscr P_{r, n}$. We fix an $r$-partition $\bnu$ such that $\blam\xrightarrow{a:i} \bnu$, where $0\le i<e$. We write $[\blam] \cup \{\gamma_1, \cdots, \gamma_a\}=\bnu$, where each $\gamma_d$ ($1\le d\le a$) is an addable $i$-node of $\blam$. For any $1\le d\le a$, $\gamma_d$ corresponds to the configuration on $L_\bs(\blam)$ where $(a_d, b_d-1)$ is occupied by a bead and $(a_d, b_d)$ is empty, with $b_d=i+c_de$.

For any $1\le d\le a$, define
$$N_i(\blam, \bnu)[d]=\Bigl(\#\add_i (\nu^{(J_{\gamma_d})} )\uparrow_{\gamma_d}-\#\rem_i (\lambda^{(J_{\gamma_d})})\uparrow_{\gamma_d}\Bigr)+ \sum^{J_{\gamma_d}-1}_{j=1}\Bigl(\#\add_i (\nu^{(j)})-\#\rem_i (\lambda^{(j)})\Bigr),$$
$N_i(\blam, \bnu)=\sum_{d=1}^a N_i(\blam, \bnu)[d].$ There are three possibilities:

{\it Case 1.} $i<\alpha$. In this case for any integer $j$ satisfying $j\equiv i\pmod e$, the columns $\col_{j-1}(\blam)$ and $\col_j(\blam)$ for $L_\bs(\blam)$ are the same as the columns $\col^+_{j-1}(\blam^{+k})$ and $\col^+_j(\blam^{+k})$ for $L_{\bs^{+k}}(\blam^{+k})$. By Lemma \ref{add and rem}, the addable and removable $i$-nodes for $\blam$ correspond exactly to the addable and removable $i$-nodes for $\blam^{+k}$.
Let $[\blam^{+k}]\cup \{\gamma^+_1, \cdots, \gamma^+_a\}=\bnu^{+k}$. For $1\le d\le a$, each $\gamma_d^+$ corresponds to the configuration on $L_\bs(\blam^{+k})$ where $(a_d, i-1+c_d(e+1))$ is occupied by a bead and $(a_d, i+c_d(e+1))$ is empty.
For any $1\le d\le a$, define
$$N_i(\blam^{+k}, \bnu^{+k})[d]=\Bigl(\#\add_i ({\nu^+}^{(J_{\gamma^+_d})} )\uparrow_{\gamma^+_d}-\#\rem_i ({\lambda^+}^{(J_{\gamma^+_d})})\uparrow_{\gamma^+_d}\Bigr)+ \sum^{J_{\gamma^+_d}-1}_{j=1}\Bigl(\#\add_i ({\nu^+}^{(j)})-\#\rem_i ({\lambda^+}^{(j)})\Bigr),$$
$N_i(\blam^{+k}, \bnu^{+k})=\sum_{d=1}^a N_i(\blam^{+k}, \bnu^{+k})[d].$
Hence, for all $1\le d\le a$, $N_i(\blam^{+k}, \bnu^{+k})[d]=N_i(\blam, \bnu)[d]$. That says, the coefficient of $s_{\bnu^{+k}}$ in $\varTheta^{(k)}\circ F_i^{(a)}(\blam)$ is the same as the coefficient of $s_{\bnu^{+k}}$ in ${^\alpha}\!\tilde{F}_i^{(a)}\circ \varTheta^{(k)}(\blam)$.
\smallskip

{\it Case 2.} $i>\alpha$. By the same argument used in Case 1 we show that  for all $1\le d\le a$, $N_i(\blam^{+k}, \bnu^{+k})[d]=N_i(\blam, \bnu)[d]$ in this case. That says, the coefficient of $s_{\bnu^{+k}}$ in $\varTheta^{(k)}\circ F_i^{(a)}(\blam)$ is the same as the coefficient of $s_{\bnu^{+k}}$ in ${^\alpha}\!\tilde{F}^{(a)}_i\circ \varTheta^{(k)}(\blam)$.

{\it Case 3.} $i=\alpha$. In this case for any integer $j$ satisfying $j\equiv i\pmod e$, the columns $\col_{j-1}(\blam)$ and $\col_j(\blam)$ for $L_\bs(\blam)$ are equal to the columns $\col^+_{j-1}(\blam^+)$ and $\col^+_{j+1}(\blam^+)$ for $L_\bs(\blam^+)$ respectively.
Let $1\le a\le r$, $b\equiv\alpha \pmod{e+1}$. By Definition \ref{Plusk}, if $b<i+k(e+1)$ then the position $(a, b)$ in $L_{\bs^{+k}}(\blam^{+k})$ has a bead; while if $b\ge i+k(e+1)$ then the position $(a, b)$ in $L_{\bs^{+k}}(\blam^{+k})$ is empty.

Any addable $i$-node of $\blam$ corresponds to a pair $(x,y)$, where $1 \le x \le r$ and $y=i+ze$, such that the position $(x, y-1)$ in the $r$-abacus $L_\bs(\blam)$ contains a bead and the position $(x, y)$ is empty.  Set $y':= i+z(e+1)$. Note that the columns $\col_{y-1}(\blam)$ and $\col_{y}(\blam)$ for $L_\bs(\blam)$ become the columns $\col_{y'-1}(\blam^{+k})$ and $\col_{y'+1}(\blam^+)$ for $L_\bs(\blam^{+k})$ after inserting the column $\overline{\rho}$. Since the position $(x, y-1)$ on $L_\bs(\blam)$ has a bead, by Lemma \ref{right}, on $L_{\bs^{+k}}(\blam^{+k})$, $(x, y'-1)$ and $(x, y')$ are both occupied by beads, $(x, y'+1)$ is empty position.
This corresponds to an addable $(i+1)$-node of $\blam^{+k}$. Similarly, one shows that any removable $i$-node $\eta$ of $\blam$ corresponds to a removable $i$-node $\eta^+$ of $\blam^{+k}$. Thus
the addable $i$-nodes and removable $i$-nodes of $\blam$ are in one-to-one correspondence with two subsets $J_{\blam^{+k},i+1}^{\rm{add}}, J_{\blam^{+k},i}^{\rm{rem}}$  of the addable $(i+1)$-nodes and removable $i$-nodes of $\blam^{+k}$ respectively. Moreover, any addable $(i+1)$-node $\tilde{\gamma}$ of $\blam^{+k}$ does not lie in $J_{\blam^{+k},i+1}^{\rm{add}}$ if and only if $\tilde{\gamma}$ corresponds to a position $(x,y')$ in $L_{\bs^{+k}}(\blam^{+k})$ such that both positions $(x, y-1)$ and $(x, y)$ in $L_\bs(\blam)$ are empty; similarly, any removable $i$-node $\tilde{\eta}$ of $\blam^{+k}$ does not lie in $J_{\blam^{+k},i}^{\rm{rem}}$ if and only if $\tilde{\eta}$ corresponds to a position  $(x,y')$ in $L_{\bs^{+k}}(\blam^{+k})$ such that both positions $(x, y-1)$ and $(x, y)$ in $L_\bs(\blam)$ are empty.

Let $1\le d\le a$, we set $$A:=\add_{i}(\blam)\uparrow_{\gamma_d},\,\,\, B:=\rem_{i}(\bnu)\uparrow_{\gamma_d}=\rem_{i}(\blam) \uparrow_{\gamma_d}. $$
Then by the discussion in the last paragraph, we have \begin{equa}\label{biiplus12}
A=\{\beta\in J_{\blam^{+k},i+1}^{\rm{add}}\mid \beta\succ{\gamma_d}\}.
\end{equa}
Let $\tilde{\bxi}$ be an $r$-partition such that $\blam^{+k}\xrightarrow{a : i+1}\tilde{\bxi} \xrightarrow{a : i} \bnu^{+k}$.
From the proof of Lemma {expansion}, we know that $\blam^{+k}\cup \{\gamma^+_1, \cdots, \gamma^+_a\}=\tilde \bxi$, $\bxi\cup \{\tilde \gamma^+_1, \cdots, \tilde \gamma^+_a\}=\bnu^{+k}$, where $\tilde \gamma_d$ is the unique addable $(i+1)$-node on $\blam^{+k}\cup \{\gamma^+_a\}$ for $1\le d\le a$.
For any $1\le d\le a$, define
$$N_i(\blam^{+k}, \tilde \bxi)[d]=\Bigl(\#\add_i (\tilde{\xi}^{(J_{\gamma^+_d})} )\uparrow_{\gamma^+_d}-\#\rem_i ({\lambda^{+k}}^{(J_{\gamma^+_d})})\uparrow_{\gamma^+_d}\Bigr)+ \sum^{J_{\gamma^+_d}-1}_{j=1}\Bigl(\#\add_i (\tilde{\xi}^{(j)})-\#\rem_i ({\lambda^{+k}}^{(j)})\Bigr),$$
$N_i(\blam^{+k}, \tilde \bxi)=\sum_{d=1}^a N_i(\blam^{+k}, \tilde \bxi)[d].$
For any $1\le d\le a$, define
$$N_i(\tilde \bxi, \bnu^{+k})[d]=\Bigl(\#\add_i ({\nu^{+k}}^{(J_{\tilde \gamma_d})} )\uparrow_{\tilde \gamma_d}-\#\rem_i (\tilde{\xi}^{(J_{\tilde \gamma_d})})\uparrow_{\tilde \gamma_d}\Bigr)+ \sum^{J_{\tilde \gamma_d}-1}_{j=1}\Bigl(\#\add_i ({\nu^{+k}}^{(j)})-\#\rem_i (\tilde{\xi}^{(j)})\Bigr),$$
$N_i(\tilde\bxi, \bnu^{+k})=\sum_{d=1}^a N_i(\tilde \bxi, \bnu^{+k})[d].$
By Lemma \ref{no+k}, $\blam^{+k}$ has no removable $(i+1)$-nodes.
Therefore, if we let $l$ denote the number of positions \((x,y)\) which lie to the right and below the position \((a_d,b_d)\) in \(L_\bs(\blam)\) and such that the positions \((x,y-1)\) and \((x,y)\) are both empty in \(L_\bs(\blam)\), then $N_{i+1}(\blam^{+k},\tilde{\bxi})=\#A+l$.
Recall that $\blam^{+k}\cup\{\gamma^+_d\}$ is the unique addable $i$-node $\{\tilde \gamma^+_d\}$. Moreover, each removable $i$-node of $\blam$ naturally corresponds to a removable $i$-node of $\blam^{+k}$ and hence corresponds to a removable $i$-node of $\tilde{\bxi}$.
Therefore, $N_{i+1}(\tilde{\bxi}, \bnu^{+k})[d]=-(\#B+l)$. Consequently, by (\ref{biiplus12}), $N_{i+1}(\blam^{+k},\tilde{\bxi})[d]+N_{i} (\tilde\bxi,\bnu^{+k})[d]=\#A-\#B=N_i(\blam, \bnu)[d]$.
Hence the coefficient of $s_{\bnu^{+k}}$ in $\varTheta^{(k)}\circ F_i(\blam)$ is the same as the coefficient of $s_{\bnu^{+k}}$ in ${^\alpha}\!F_i\circ \varTheta^{(k)}(\blam)=\tilde{F}_{i}^{(a)}\tilde{F}_{i+1}^{(a)} \blam^{+k}$.
This completes the proof of the proposition.
\end{proof}

\begin{lemma}\label{5.10}
Fix $k\in\N, \bs\in \overline{\mathcal A}^r_e$ and an integer $0\leq \alpha<e$. For any $r$-partitions $\blam,\bmu$, we have $\blam \prec \bmu$ if and only if $\blam^{+k} \prec \bmu^{+k}$.
\end{lemma}

\begin{proof} $\blam \prec \bmu$ if and only if there exist $j\in \{1,\ldots,r\}$ and $k\in \mathbb{Z}_{>0}$ such that $\lambda^{(a)}=\mu^{(a)}$ for all $a\in \{1,\ldots,j-1\}$, $\lambda^{(j)}_m=\mu^{(j)}_m$ for all $m\in \{1,\ldots,l-1\}$, and $\lambda^{(j)}_l<\mu^{(j)}_l$.

Recall $\CIRCLE_j^i(\blam, \bs)$ denotes the $j$-th bead on the abacus $L_{s_i}(\blam^{(i)})$, counted from right to left. By definition, $\CIRCLE_j^i(\blam, \bs)$ is located in column $s_i+\lambda^{(i)}_j-j$ on the $i$-th runner $L_{s_i}(\blam^{(i)})$ of the $r$-abacus $L_\bs(\blam)$. Hence, $\blam \prec \bmu$ if and only if $L_{s_a}(\blam^{(a)})=L_{s_a}(\bmu^{(a)})$ for all $a\in \{1, \ldots, j-1\}$, $\CIRCLE^j_m$ of $L_{s_j}(\lambda^{(j)})$ and that of $L_{s_j}(\mu^{(j)})$ are in the same column for all $m\in \{1, \ldots, l-1\}$, and $\CIRCLE^j_l$ of $L_{s_j}(\lambda^{(j)})$ lies strictly to the left of $\CIRCLE^j_l$ of $L_{s_j}(\mu^{(j)})$.

Without loss of generality, we may assume that, in passing from $L_\bs(\bmu)$ to $L_{\bs^{+k}}(\bmu^{+k})$, $c$ beads are added to the right of $\CIRCLE^j_l$ in $L_{s_j}(\mu^{(j)})$.
Note that the beads inserted in passing from $L_\bs(\blam)$ to $L_{\bs^{+k}}(\blam^{+k})$ and from $L_\bs(\bmu)$ to $L_{\bs^{+k}}(\bmu^{+k})$ lie in the same columns. Therefore, $L_{s^{+k}_a}((\lambda^{+k})^{(a)})=L_{s^{+k}_a}((\mu^{+k})^{(a)})$ for all $a\in \{1, \ldots, j-1\}$, the bead $\CIRCLE^j_m$ of $L_{s^{+k}_j}\bigl((\lambda^{+k})^{(j)}\bigr)$ and that of $L_{s^{+k}_j}\bigl((\mu^{+k})^{(j)}\bigr)$ lie in the same column for all $m\in \{1, \ldots, l+c-1\}$, and the bead $\CIRCLE^j_{l+c}$ of $L_{s^{+k}_j}\bigl((\lambda^{+k})^{(j)}\bigr)$ lies to the left of the bead $\CIRCLE^j_{l+c}$ of $L_{s^{+k}_j}\bigl((\mu^{+k})^{(j)}\bigr)$. It follows that $\blam \prec \bmu$ if and only if $\blam^{+k} \prec \bmu^{+k}$.
\end{proof}

\begin{proposition}\label{KeyCanonicalBases2}
Fix $k\in\N, \bs\in \overline{\mathcal A}^r_e$ and an integer $0\leq \alpha<e$. Let $\bmu$ be an $(e,\bs)$-regular $r$-partition. Suppose $(k-1)e \ge s_r+n$. Then $G^{\bs^{+k}}_{e+1}(\bmu^{+k})=\varTheta^{(k)}(G^\bs_e(\bmu))$.
\end{proposition}

\begin{proof} Applying Propositions \ref{PlusRegular2}, \ref{diagram commute2} and Lemma \ref{5.10}, one proves the proposition by exactly the same argument as in the proof of Proposition \ref{KeyCanonicalBases}.
\end{proof}

\medskip
\noindent
{\textbf{Proof of Theorem B}: } This follows from Propositions \ref{PlusRegular2}, \ref{KeyCanonicalBases2} and Theorem \ref{A1}.\hfill\qed
\medskip

Note that the cyclotomic Hecke algebra $\H_{r,n}^{\Lam_\bs}$ depends only on $e$ and the multi-set $\{s_j+e\Z\mid 1\leq j\leq r\}$ but not on $\{s_j\mid 1\leq j\leq r\}$. Given any cyclotomic Hecke algebra $\H_{r,n}^{\Lam_\bs}$,
we can always replace certain $s_j$ by $s_j+te$ for some $1\leq j\leq r$ and $t\in\Z$ so that $\bs=(s_1,\cdots,s_r)$ satisfies that $\bs\in \overline{\mathcal A}^r_e$ and $(k-1)e \ge s_r+n$ without changing the algebra $\H_{r,n}^{\Lam_\bs}$ and the modules $S^\blam, D^\bmu$.

\begin{example} Suppose that $e=4, \bs=(0,2), k=3, r=2, n=4, \alpha=0$, $\blam=((1^2),(1^2)), \bmu=((2),(2))$, $(k-1)e=(3-1)e\ge 2+4=s_r+n$. Then we have that $$
\blam^{+k}=((8, 4, 1^3),(6, 2, 1^2)),\,\,\quad\bmu^{+k}=((8, 4,2, 1),(6, 2^2)), $$
and $$
\bigl[\hat{S}^{((8, 4, 1^3),(6, 2, 1^2)))}:\hat{D}^{((4, 2, 1),(2^2))}\bigr]=\bigl[S^{((1^2),(1^2))}:D^{((2),(2))}\bigr]=1.
$$
\end{example}

\section{Declarations}

The authors did not use AI or LLM tools for any aspect of this research or the
writing of this manuscript.

\bigskip

   .


\begin{thebibliography}{99}

\bibitem{A1} {\sc S.~Ariki}, {\em On the decomposition numbers of the Hecke algebra of $G(m,1,n)$}, J. Math. Kyoto Univ., {\bf 36} (1996), 789--808.

\bibitem{AK} {\sc S.~Ariki and K.~Koike}, {\em  A Hecke algebra of $(\Z/r\Z)\wr\Sym_n$ and construction of its representations}, Adv. Math., {\bf 106} (1994), 216--243.

\bibitem{AM} {\sc S.~Ariki, A.~Mathas}, {\em The number of simple modules of the Hecke algebras of type $G(r,1,n)$}, Math. Zeit., {\bf 233}(3), (2000), 601--623.

%\bibitem{BCG}
%{\sc O.~Brunat, N.~Chapelier-Laget and T.~Gerber}, {\em Generalised core partitions and Diophantine equations}, arXiv: 2403.11191.


\bibitem{BM:cyc}
{\sc M.~Brou{\'e} and G.~Malle}, {\em Zyklotomische {H}eckealgebren}, in: Repr{\'e}sentations Unipotentes G{\'e}n{\'e}riques et Blocs des Groupes R{\'e}ductifs Finis,
  Ast\'erisque., {\bf 212} (1993), 119--189.

\bibitem{BrundanKleshchev09}
{\sc J.~Brundan and A.~Kleshchev}, {\em Blocks of cyclotomic Hecke algebras and Khovanov--Lauda algebras}, Invent. Math., {\bf 178}no.~3, (2009), 451--484.

\bibitem{BK:GradedDecomp}
{\sc J.~Brundan and A.~Kleshchev}, {\em Graded decomposition numbers for cyclotomic Hecke algebras}, Adv. Math., {\bf 222} (2009), 1883--1942.

\bibitem{C} {\sc I.V.~Cherednik}, {\em A new interpretation of Gelfand-Tzetlin bases}, Duke Math. J., {\bf 54}(2) (1987), 563--577.

\bibitem{DellA24b}
{\sc A.~Dell'Arciprete}, {\em Full runner removal theorem for {A}riki-{K}oike algebras}, J. Algebra, {\bf 660} (2024), 513--563.

\bibitem{DP2026}
{\sc A.~Dell'Arciprete, L.~Putignano}, {\em Empty runner removal theorem for Ariki-Koike algebras}, J. Algebra, {\bf 698} (2026), 316--348.

\bibitem{DJ}
{\sc R.~Dipper and G.~James}, {\em $q$-tensor space and $q$-Weyl modules}, Trans. Amer. Math. Soc., {\bf 327} (1991), 251--282.


\bibitem{DJM} {\sc R.~Dipper, G.D.~James and A.~Mathas}, {\em Cyclotomic $q$-Schur algebras},  Math. Zeit., {\bf 229}(3) (1998), 385--416.

\bibitem{DM} {\sc R.~Dipper and A.~Mathas}, {\em Morita equivalences of Ariki-Koike algebras}, Math. Z., {\bf 240} (2002), 579--610.

\bibitem{F1} {\sc M.~Fayers}, {\em Weights of $r$-partitions and representations of Ariki-Koike algebras}, Adv. Math., \textbf{206} (2008) 112--144.

\bibitem{F2} {\sc M.~Fayers}, {\em Core blocks of Ariki-Koike algebras}, J. Algebr. Comb., {\bf 26} (2007), 47--81.

%\bibitem{F2007} {\sc M.~Fayers}, {\em James's Conjecture holds for weight four blocks of Iwahori-Hecke algebras}, J. Algebra, 317 (2007) 593--633.

\bibitem{Fay08} {\sc M.~Fayers}, {\em Another runner removal theorem for $v$-decomposition numbers of Iwahori-Hecke algebras and $q$-Schur algebras}, {\bf 310} (2007), 396--404.

%\bibitem{F2008} {\sc M.~Fayers}, {\em Decomposition numbers for weight three blocks of symmetric groups and Iwahori-Hecke algebras},
%Trans. Amer. Math. Soc., {\bf 360} (2008), 1341--1376.
%
%\bibitem{Fay10} {\sc M.~Fayers}, {\em An {LLT}-type algorithm for computing higher-level canonical bases}, J. Algebra, {\bf 214} (2010), 2186--2198.

\bibitem{GL} {\sc J.~Graham and G.~Lehrer}, {\em Cellular algebras}, Invent. Math., {\bf 123} (1996), 1--34.

\bibitem{HM}
{\sc J.~Hu and A.~Mathas}, {\em  Graded cellular bases for the cyclotomic Khovanov-Lauda-Rouquier algebras of type $A$}, Adv. Math., {\bf 225} (2010), 598--642.

\bibitem{HHLQ} {\sc W.~Hu, F.~Huang, Y.~Li and X.~Qi}, {\em Level-rank dualities and moving vectors}, preprint, {arXiv:2604.25340}, 2026.

\bibitem{J} {\sc N.~Jacon}, {\em Kleshchev $r$-partitions and extended Young diagrams}, Adv. Math., {\bf 339} (2018), 367--403.

\bibitem{JL} {\sc N.~Jacon and C.~Lecouvey}, {\em Cores of Ariki-Koike algebras}, Doc. Math., {\bf 26} (2021), 103--124.

%\bibitem{James}
%{\sc G.D.~James}, {\em Some combinatorial results involving Young diagrams}, Proc. Camb. Phil. Soc., {\textbf{83}} (1978), 1--10.

\bibitem{JK}
{\sc G.D.~James and A.~Kerber}, {\em The representation theory of the symmetric group}, Encyclopedia of Mathematics and its Applications, {\textbf{16}}, Addison-Wesley, 1981.

\bibitem{JM}
{\sc G.D.~James and A.~Mathas}, {\em Equating decomposition numbers for different primes}, {J. Algebra},  {\textbf{258}} (2002), 599--614.

\bibitem{LLT96}
{\sc A.~Lascoux, B.~Leclerc, and J.-Y.~Thibon}, {\em Hecke algebras at roots of unity and crystal bases of quantum affine
  algebras}, Comm. Math. Phys., {\bf 181} (1996), 205--263.

\bibitem{LQ2025} {\sc Y.~Li and X.~Qi}, {\em Moving vectors I: Representation type of blocks of Ariki-Koike algebras},
J. Lond. Math. Soc., {\bf 111}(5), (2025), e70169.

\bibitem{LQT2026} {\sc Y.~Li, X.~Qi and K.~Tan}, {\em Moving vectors and core blocks of Ariki-Koike
algebras}, {J. Algebra}, {\bf 694} (2026), 497-563.

\bibitem{LT}
{\sc Y.~Li and K.~Tan}, {\em Cores and weights of multipartitions and blocks of Ariki-Koike algebras}, {J. Algebr. Comb.} {\bf 61} (2025), 43.

\bibitem{LZZ}
{\sc Y.~Li, J.~Zhang, and S.~Zhu}, {\em Core abaci and Diophantine equations I: fundamental weight}, preprint, arXiv:2604.23311, 2026.

\bibitem{LM} {\sc S.~Lyle and A.~Mathas}, {\em Blocks of cyclotomic Hecke algebras}, Adv. Math., {\bf 216} (2007), 854--878.

\bibitem{Ma} {\sc A.~Mathas},   {\em Matrix units and generic degrees for the Ariki-Koike algebras}, {J. Algebra,} {\bf 281} (2004), 695--730.

\bibitem{Q} {\sc T.~Qin},   {\em Subdivision and runner removal theorems}, {J. Algebra,} {\bf 711} (2027), 590--675.


\bibitem{U} {\sc D.~Uglov}, {\em Canonical bases of higher level $q$-deformed Fock spaces and Kazhdan-Lusztig polynomials}, in Physial
Combinatorics (ed. M. Kashiwara, T. Miwa), Progress in Math., {\bf 191}, Birkhauser, (2000).

\bibitem{VV} {\sc M.~Varagnolo, E.~Vasserot}, {\em On the decomposition matrices of the quantized Schur algebra}, Duke
J. Math., {\bf 100} (1999), 267--297.
\end{thebibliography}
\end{document}